\documentclass{amsart}
\usepackage{amsmath}
\usepackage{amsthm}
\usepackage{amssymb}
\usepackage{amsaddr}
\usepackage{mathrsfs}
\usepackage{enumerate}
\usepackage{graphicx}
\usepackage{fullpage}

\newcommand{\sL}{\mathscr L}
\newcommand{\sJ}{\mathscr J}
\newcommand{\p}[1]{\left( {#1} \right)}
\newcommand{\norm}[1]{\left\| #1 \right\|}

\newcommand{\br}[1]{\left[ #1 \right]}
\newcommand{\Br}[1]{\left\{ #1 \right\}}

\newcommand{\dual}[1]{\left< #1 \right>}
\newcommand{\trace}[1]{\texttt{trace}\left( #1\right)}

\newcommand{\sP}{\mathcal{P}}
\newcommand{\bA}{\mathbf A}
\newcommand{\bX}{\mathbf X}
\newcommand{\bQ}{\mathbf Q}
\newcommand{\bG}{\mathbf G}
\newcommand{\ARE}{\bA \bX + \bX \bA^* - \bX \bG \bX + \bQ}
\newcommand{\bPhi}{\boldsymbol{\Phi}}
\newcommand{\bS}{\mathbf S}
\newcommand{\bI}{\mathbf I}
\newcommand{\bT}{\mathbf T}
\newcommand{\bLambda}{\boldsymbol{\Lambda}}
\newcommand{\bd}{\mathbf d}

\newcommand{\bPsi}{\boldsymbol{\Psi}}

\newcommand{\bW}{\mathbf W}

\newtheorem{theorem}{Theorem}
\newtheorem{lemma}{Lemma}
\newtheorem{corollary}{Corollary}

\newtheorem{remark}{Remark}

\newtheorem{definition}{Definition}
\title{Penalized Weighted Trace Minimization for Optimal Control Device Design and Placement}
\author{James Cheung}
\thanks{The material presented in this paper was developed independently by the author and is not associated with any work related to his affiliated organization. }
\address{Toyon Research Corporation, 6800 Cortona Drive, Goleta, CA 93117. }

\begin{document}

\maketitle
\begin{abstract}
In this paper, we present a new analytical framework for determining the well-posedness of constrained optimization problems that arise in the study of optimal control device design and placement within the context of infinite dimensional linear quadratic control systems. We first prove the well-posedness of the newly minted "strong form" of the time-independent operator-valued Riccati equation. This form of the equation then enables the use of trace-class operator analysis and the Lagrange multiplier formalism to analyze operator-valued Riccati equation-constrained optimization problems. Using this fundamental result, we then determine the conditions under which there exists unique solutions to two important classes of penalized trace minimization problems for optimal control device placement and design.
\end{abstract}

\section{Introduction}
The purpose of the work is to address the fundamental question of well-posedness for optimization problems associated with optimal sensor and actuator placement and design in the context of infinite-dimensional linear systems theory. To consolidate definitions, we will collectively refer to sensors and actuators as \emph{control devices}. In the context of the linear quadratic regulator (LQR), the optimization problem of interest is the following:
$$
    \min_{p \in \sP} \Br{ \min_{u(\cdot)\in \sL(\mathbb R_+;U)} \int_0^{+\infty}
    \br{\p{z(t),\bQ z(t)}_H + \p{u(t;p),\mathbf R u(t;p)}_U} dt}
$$
subject to
$$
\left\{
\begin{aligned}
    \frac{\partial z}{\partial t} &= \bA z(t) + \mathbf B(p)u(t;p) \\
    z(0) &= z_0
\end{aligned}
\right.,
$$
for all $t \in [0, \infty)$; where denoting $H$, $U$, and $\sP$ as separable Hilbert spaces associated with the state, control, and parameter respectively, we define $p\in \sP$ to be the generalized design parameter (e.g. actuator placement or geometric design variables), $z(\cdot)\in H$ to be the state variable, $u(\cdot;p) \in L^2(\mathbb R_+; H)$ to be the control variable, $\bQ \in \sL^s(H)$ to be the output operator, $\mathbf R \in \sL(U)$ to be the control weighting operator, $\mathbf B(p) \in \sL(U;H)$ to be the parametrized control operator associated with the control device (e.g. sensors or actuators), and $\bA: \mathcal D(\bA) \rightarrow H$ to define the state process as the generator of a $C_0$-semigroup. The Dynamic Programming Principle determines that the optimal control, for a fixed value $p \in \sP$, is given by
$$
    u_{opt}(t;p) = -\mathbf R^{-1} \mathbf B^*(p) \bX(p) z(t),
$$
where $\bX(p) \in \sL(H)$ is the solution to the weak form operator-valued Riccati equation
$$
    \p{\phi,\br{\bA^*\bX(p) + \bX(p)\bA - \bX(p)\mathbf B(p) \mathbf R^{-1}\mathbf B(p)\bX(p)_ + \bQ}\psi}_H = 0.
$$
for all $\phi,\psi \in \mathcal D(\bA)$. It then follows \cite[Theorem 6.2.4]{curtain2012introduction} that
$$
\begin{aligned}
\min_{u(\cdot)\in L(\mathbb R_+;U)} \int_0^{+\infty}
    \br{\p{z(t),\bQ z(t)}_H + \p{u(t;p),\mathbf R u(t;p)}_U} dt 
    &= \p{z_0 \bX(p) z_0}_H \\
   &=\trace{\bX(p) \bW},
\end{aligned}
$$
where $\bW(\cdot) := z_0\p{z_0, \cdot}_H$ is the operator generated through the exterior product using the initial condition of the system $z_0 \in H$. With these combined observations, the optimal actuator placement and design problem becomes
$$
    \min_{p \in \sP} \trace{\bX(p) \bW}
$$
subject to
$$  
    \p{\phi,\br{\bA^*\bX(p) + \bX(p)\bA - \bX(p)\mathbf B(p) \mathbf R^{-1}(p)\mathbf B(p)^*\bX(p) + \bQ}\psi}_H = \mathbf 0,
$$
for all $\phi,\psi \in \mathcal D(\bA)$. This same optimization problem, albeit with $\bA$ and $\bA^*$ switched in the statement of the operator-valued Riccati equation carries over to the setting of optimal sensor placement and design for linear state estimation systems (i.e. in the context of the K\'alm\'an filter). Therefore, the study of operator-valued Riccati equation constrained weighted trace minimization problems has broad reaching implications in designing optimal systems for both control and state estimation. 

The problem of optimal device placement and design has its origins in the work of \cite{bensoussan2006optimization}, where the notion of optimal sensor placement has been introduced for infinite-dimensional systems. This theory was then extended to actuator placement in LQR systems in \cite{5482053}. Additional notable extensions of the theory have been made to device placement in $\mathcal H_\infty$ control systems \cite{hintermuller2017optimal} and joint parameter estimation and sensor placement in partially observed systems \cite{sharrock2022joint}. The theory of optimal device placement and design has been applied to many practical problems including optimizing thermal control \cite{hu2016sensor}, vibration damping \cite{morris2015comparison}, and mobile sensor \cite{burns2015infinite} systems. A general observation made when studying the existing literature on optimal device placement and design indicates that a constrained optimizer $p_{opt} \in \sP$ that minimizes the weighted operator trace of $\bX(p)$ can be shown to exist; however uniqueness of the optimizer is almost always left as an open problem. This is the key motivator for writing this paper. 

We approach the problem of determining the conditions by which a unique solution to the constrained trace minimization problem associated with optimal device placement and design through penalization techniques. The first problem we study is the control penalization problem, i.e., to seek a constrained minimizer to 
$$
    \mathcal J_\beta(p) = \trace{\bX(p)\bW} + \frac{\beta}{2} \norm{p}^2_\sP,
$$
where we have introduced the penalization parameter $\beta \in \mathbb R_+$ to regularize the cost functional. This penalization scheme is a classical technique used to improve the conditioning of optimization problems that arise in inverse problems \cite{kirsch2011introduction} and optimal control \cite{troltzsch2010optimal}. We demonstrate in \S\ref{sec: problem 1} that the expected result of choosing $\beta$ sufficiently large induces uniqueness of the constrained minimizer $p_{opt} \in \sP$. 

The study of the first penalized optimization problem will form the pedagogical basis for the second problem discussed in this work: the determination of conditions by which an unique optimizer $p_{opt} \in \sP$ exists on constraint manifolds defined by
$$
    \trace{\mathbf B(p)\mathbf R^{-1}(p) \mathbf B^*(p)} = \gamma,
$$
where $\gamma \in \mathbb R_+$ is a positive constant. This constraint on $\mathbf B(p)\mathbf R^{-1} \mathbf B^*(p)$ serves the practical purpose of constraining the so-called gain of the control device, e.g. the amount of control effort used by the control feedback law. Following the theme of penalization, this additional trace constraint is approximately enforced through the following penalized cost functional
$$
    \mathcal J_\beta(p) = \trace{\bX(p)\bW} + \frac{\beta}{2}\br{\trace{\mathbf B(p) \mathbf R^{-1} (p) \mathbf B^*(p)} - \gamma}^2.
$$
Exact constraint enforcement is achieved in the limit as $\beta \rightarrow +\infty$. We determine the conditions under which there exists a unique constrained minimizer to this cost functional in \S\ref{sec: problem 2}. The penalization introduced in the second problem serves primarily as a mechanism to determine uniqueness on the additional constraint manifold. 

The constrained optimization problems are cast into the Lagrange multiplier formalism to study the two penalized constrained trace minimization problems posed in the previous section. A fixed-point argument (c.f. the Banach fixed point theorem \cite[Theorem 3.7-1]{ciarlet2025linear}) is used to determine the uniqueness of the solution to the associated first-order optimality system associated with the Lagrangian functional for each problem. The second-order sufficient optimality condition is then applied to determine that the unique solution of the first-order optimality system is indeed a minimizer. 

While this strategy is sound at first glance, we quickly run into technical issues when attempting to apply the Lagrange multiplier formalism to the operator-valued Riccati equation. The problem is that a sufficiently ``strong'' form of this equation, i.e. of the form
$$
    \bA^*\bX + \bX\bA-\bX\mathbf B\mathbf R^{-1} \mathbf B^*\bX + \bQ = \mathbf 0,
$$
has not been shown to be well-posed on all of the Hilbert space $H$ associated with the state of the system. This inhibits the use of the appropriate trace-class operator analytic tools required to rigorously derive the first-order optimality system associated with the constrained optimization problems discussed in this work. The prior literature (see e.g. \cite{bensoussan2007representation}) only indicates that the operator-valued Riccati equation (without appealing to its Bochner integral form \cite{burns2015solutions}) is only well-defined as an operator equation whose domain is defined on $\mathcal D(\bA)$. However, we are able to overcome this technical hurdle in this work by determining that the ``strong'' form of the operator-valued Riccati equation is in fact well posed in Theorem \ref{theorem: ARE well-posedness}. 

The structure of the paper is the following: First, we begin in \S \ref{section: Notation} with the notation that will be used throughout the work, then in \S\ref{sec: primal and dual} we analyze the well-posedness of the strong form of the operator-valued Riccati equations and its associated dual problem along with the Lipschitz continuity of their solutions with respect to the control parameter $p\in\sP$. Next, in \S\ref{sec: problem 1} and \S\ref{sec: problem 2} we analyze the penalized problems of interest, and finally we conclude this paper with a discussion of our findings in \S\ref{sec: discussion}. 

\section{Notation} \label{section: Notation}
Let $H$ be a separable complex Hilbert space with its  inner product denoted as $\p{\cdot,\cdot}_H : H\times H \rightarrow \mathbb R_+$. Throughout this work, we will define $\Br{e_i}_{i=1}^\infty$ an orthonormal basis of $H$. This means that any element $\phi \in H$ can be represented as
$$
    \phi = \sum_{i=1}^\infty c_i e_i,
$$
where $\Br{c_i}_{i=1}^\infty \in \mathbb C$ are scalar coefficients, and that $\p{e_i, e_j}_H = \delta_{ij}$ with 
$$
\delta_{ij} :=\begin{cases}
1 & \textrm{if } i=j \\
0 & \textrm{Otherwise}
\end{cases}
$$ 
denoting the Kronecker delta function. We denote the space of bounded linear operators mapping $H$ onto $H$ as $\sL\p{H}$. where the $\sL\p{H}$ norm is defined by
$$
   \norm{\bT}_{\sL\p{H}} := \sup_{\substack{\phi \in H\\\phi \neq 0}} \frac{\norm{\bT\phi}_H}{\norm{\phi}_H}.
$$
For any $\bT \in \sL\p{H}$, the definition of the adjoint operator $\bT^* \in \sL\p{H}$ is defined through the inner-product by the following
$$
    \p{\phi, \bT\psi}_H = \p{\bT^*\phi, \psi}_H
$$
for all $\phi, \psi \in H$. 

In this work, we are also be interested in the space of bounded linear operators on Banach spaces. Let $V_1$ and $V_2$ be two complex Banach spaces. We then denote the space of bounded linear operators mapping $V_1$ to $V_2$ as $\sL\p{V_1; V_2}$. The $\sL\p{V_1; V_2}$ norm is then defined by
$$
    \norm{\bT}_{\sL\p{V_1;V_2}} := \sup_{\substack{\phi \in V_1\\ \phi \neq 0}} \frac{\norm{\bT \phi}_{V_2}}{\norm{\phi}_{V_1}}
$$
for all $\phi \in V_1$. With the basic functional analytic notation defined, we now move on to define the more technical notions needed for this work.  

\subsection{Trace-Class Operators}
Trace class operators generalize the notion of finite dimensional matrices with finite trace (i.e. the matrix Lie algebra $\mathfrak{gl} (n;\mathbb C)$) to the setting of infinite dimensional linear operators. The formal definition for the space of trace-class operators $\sJ_1(H)\subset \sL(H)$ is given by
$$
    \sJ_1(H) := \Br{\bT \in \sL(H) : |\trace{\bT}| < \infty},
$$
where the operator trace $\trace{\cdot}: \sJ_1(H) \rightarrow \mathbb C$ is defined by 
$$
    \trace{\bT} := \sum_{k=1}^\infty \p{e_k, \bT e_k}_{H}
$$
for all $\bT \in \sJ_1(H)$, where $\Br{e_k}_{k=1}^\infty$ again forms an orthonormal basis of the Hilbert space $H$. The $\sJ_1(H)$ norm is then defined by
$$
    \norm{\bT}_1 := |\trace{\bT}|,
$$
where $|\cdot|:\mathbb C \rightarrow \mathbb R_+$ denotes the modulus on the field $\mathbb C$. From \cite[Theorem 18.11]{conway2025course}, we have that $\sJ_1(H)$ is a two-sided *-ideal in $\sL(H)$, meaning that for any $\mathbf U \in \sL(H)$ and $\mathbf V \in \sJ_1(H)$, we have that $\mathbf U \mathbf V \in \sJ_1(H)$ and $\mathbf V\mathbf U \in \sJ_1(H)$. 

Using the operator trace and the fact that $\sJ_1(H)$ is a two-sided *-ideal in $\sL(H)$, we are able to induce the following definition of the duality pairing $\dual{\cdot,\cdot}: \sL(H)\times\sJ_1(H) \rightarrow \mathbb C$ as follows
$$ 
    \dual{\mathbf U, \mathbf V} 
    := \trace{\mathbf U^* \mathbf V}
$$
for all $\mathbf U \in \sL(H)$ and $\mathbf V \in \sJ_1(H)$. There is a one-to-one correspondence (i.e. an isometric isomorphism) between $\dual{\mathbf U, \cdot}: \sJ_1(H) \rightarrow \mathbb C$ and $\sJ_1(H)'$, the dual space of $\sJ_1(H)$ \cite[Theorem 19.2]{conway2025course}. From the definition of the operator and the invariance of conjugation under the trace operation, we have that
\begin{equation}
\label{eqn: operator adjoint trace identity}
    \dual{\mathbf U, \mathbf V} 
    = \dual{\bI, \mathbf U^* \mathbf V} 
    = \dual{\bI, \mathbf V^* \mathbf U} 
    = \dual{\mathbf V^*, \mathbf U^*},
\end{equation}
where we have denoted $\bI$ as the identity element in $\sL\p{H}$. This identity will be utilized frequently in the derivation of the first-order optimality conditions associated with the penalized optimization problems studied in \S \ref{sec: problem 1} and \S\ref{sec: problem 2}. 

\subsubsection{Symmetric Operators}
Of particular interest in this work is the subspace of symmetric operators in $\sJ_1\p{H}$ and $\sL\p{H}$. An operator $\bT\in\sL\p{H}$ is symmetric if
$$
    \p{\phi, \bT \psi}_H = \p{\phi, \bT^* \psi}_H
$$
for all $\phi, \psi \in H$. The subspace of all symmetric operators in $\sL\p{H}$ will be denoted as $\sL^s\p{H}$. We will then define $\sJ_1^s(H):=\sJ_1(H)\cap\sL(H)$ as the space of symmetric operators in $\sJ_1(H)$. The definition of the norm for the spaces $\sL^s(H)$ and $\sJ_1^s(H)$ coincides with the $\sL(H)$ and $\sJ(H)$ norms respectively. 

Semi-definite operators arise frequently during the course of the discussion presented in this work. A symmetric operator $\bT \in \sL^s(H)$ is \emph{positive semi-definite} if
$$
    \p{\phi, \bT \phi}_H \geq 0
$$
for all $\phi \in H$. A symmetric operator is then \emph{positive definite} if the inequality is strict. 

\subsection{Exponentially Stable $C_0$-Semigroups}

Following \cite[\S 2]{goldstein2017semigroups}, we define a \emph{$C_0$-semigroup} as a one-parameter family of operators $\Br{\bS(t) \in \sL(H): t \in \mathbb R_+\cup{0}}$ that satisfies the following
\begin{enumerate}[i)]
    \item $\bS(t)\bS(s) = \bS(t+s)$ for each $t,s \in \mathbb R_+$,
    \item $\bS(0) = \mathbf I$, and
    \item $\bS(t)\phi \in H$ is norm-continuous with respect to $t\in\mathbb R_+$ for all $\phi \in H$.
\end{enumerate}
A $C_0$-semigroup is then said to be \emph{exponentially stable} if there exists positive constants $M,\alpha \in \mathbb R_+$ satisfying
\begin{equation} \label{eqn: exponential stability bound}
    \norm{\bS(t)}_{\sL(H)} \leq Me^{-\alpha t}
\end{equation}
for any $t \in \mathbb R_+$. An (unbounded) operator $\bA$ is said to be a \emph{generator} of $\bS(t)$ if 
\begin{equation} \label{eqn: A definition}
\bA\phi = \lim_{h\rightarrow 0^+} h^{-1}\br{\bS(h)\phi - \phi} \in H
\end{equation}
for all $\phi \in \mathcal D(\bA)$, where
$$
    \mathcal D(\bA):=\Br{\phi \in H: \norm{\bA\phi}_H < \infty}. 
$$
The analysis provided in the proof of \cite[Theorem 2.6]{goldstein2017semigroups} indicates that $\mathcal D(\bA)$ is densely defined in $H$. 

The adjoint of $\bS(t)$, denoted by $\bS^*(t)$, is also a bounded operator on $H$. This is easily observed by the following
$$
    \p{\phi, \bS(t)\psi}_H =\p{\bS^*(t)\phi, \psi}_H
$$
for all $\phi,\psi \in H$ and all $t \in \mathbb R_+$. This identity indicates also that
$$
    \norm{\bS^*(t)}_H \leq Me^{-\alpha t}
$$
through the induced operator norm. From \cite[Thoerem 4.3]{goldstein2017semigroups}, we have that
$\bA^*$, the adjoint operator of $\bA$, is the generator of $\bS^*(t)$, i.e.
\begin{equation} \label{eqn: A* def}
\bA^*\psi = \lim_{h \rightarrow 0^+} h^{-1}\br{\bS^*(h) \psi - \psi}\in H
\end{equation}
for all $\psi \in \mathcal D(\bA^*)$, where we have defined
$$
\mathcal D(\bA^*) := \Br{\phi \in H: \lim_{h \rightarrow 0^+} h^{-1}\p{\bS^*(h)\phi - \phi}  \in H}.
$$
$\mathcal D(\bA^*)$ is also a densely defined subset of $H$. 

Bounded perturbations to a generator of an exponentially stable $C_0$-semigroup also generates a semigroup, i.e. if $\bA: \mathcal D(\bA) \rightarrow H$ generates a semigroup $\bS(t)$, then $\bA-\bT$, where $\bT \in \sL\p{H}$ is a bounded positive semi-definite operator, also generates a $C_0$-semigroup \cite[Theorem 6.4]{goldstein2017semigroups}. The perturbation semigroup generated by $\bA-\bT$ is then also exponentially stable since $\bT\in \sL\p{H}$ is positive semi-definite.

\section{The Operator-Valued Riccati Equation and Its Dual Equation}
\label{sec: primal and dual}

This section is dedicated to the study of the strong-form of the operator-valued Riccati equation and its dual equation that arises from the derivation of the first-order optimality system associated with the constrained optimization problems discussed in \S\ref{sec: problem 1} and \S\ref{sec: problem 2}. The focus of this section is on determining the well-posedness and the Lipschitz continuity (with respect to varying control parameters $p \in \sP$) of these equations. In this section and the remainder of this work, we will take $\bG:= \mathbf B\mathbf R^{-1} \mathbf B^*$ and $\bG_p$ to be its parametrized analog to simplify notation. 

\subsection{Strong Operator-Valued Riccati Equation}

We now motivate the definition of the \emph{strong form} of the operator-valued Riccati equation. This form is essential in defining the Lagrangian first-order optimality system that we use to determine the well-posedness of the penalized weighted trace minimization problems studied in this work. We begin with the following.

\begin{definition}
    A symmetric positive semi-definite operator $\bX \in \sL^s(H)$ is said to be a solution of the \emph{strong operator-valued Riccati equation} if it satisfies
    \begin{equation} \label{eqn: strong ARE}
        \ARE = \mathbf 0 
    \end{equation} 
    in the $\sL(H)$ topology (i.e. $\norm{\ARE}_{\sL(H)} = 0$) with the additional condition that $\bA\bX + \bX\bA^* \in \sL^s(H)$, where $\bA: \mathcal D(\bA) \rightarrow H$ is the generator of an exponentially stable $C_0$-semigroup semigroup and the coefficient operators $\bG,\bQ \in \sL^s(H)$ are symmetric positive semi-definite.
\end{definition}

In \cite[Chapter IV-1 Section 3]{bensoussan2007representation} the notion of \emph{strict} and \emph{classical} solutions are presented to describe solutions to \eqref{eqn: strong ARE}. The analysis regarding these solutions was done in the context of $\bA\bX+\bX\bA^*$ being an operator on $\mathcal D(\bA^*)$. In contrast, we have determine in Theorem \ref{theorem: ARE well-posedness} that $\bA\bX + \bX\bA^*$ is actually a bounded operator on all of $H$. This finding opens up the possibility of utilizing trace-class operator theory in the analysis of operator-valued Riccati equations without reformulating it as a Bochner integral equation as done in \cite{burns2015solutions}. This is the key result that enables the derivation of the first-order optimality systems utilized in this work. 

We now determine that the strong form of the operator-valued Riccati equation \eqref{eqn: strong ARE} is well-defined and is also equivalent to the Bochner integral form of the operator-valued Riccati equation (c.f. \cite{burns2015solutions}) in the following. It is further determined that the solution to \eqref{eqn: strong ARE} is a trace-class operator if $\bQ \in\sJ_1^s(H)$. 

\begin{theorem} \label{theorem: ARE well-posedness}
Assume that $\bA : \mathcal D(\bA) \rightarrow H$ is the generator of an exponentially stable $C_0$-semigroup $\bS(t)\in \sL(H)$, $\bQ, \bG \in \sL^s(H)$ are symmetric positive semi-definite operators. Then the unique positive semi-definite solution $\bX \in \sL^s(H)$ to following Bochner integral equation
\begin{equation} \label{eqn: Bochner ARE}
    \bX = \int_0^{+\infty} \bS(t)\p{\bQ - \bX \bG \bX}\bS^*(t)dt
\end{equation}
is the unique solution to \eqref{eqn: strong ARE}. 

Furthermore, if we assume that $\bQ \in \sJ^s_1(H)$, then we have that $\bX \in \sJ_1^s(H)$ is a trace class operator that satisfies the following
\begin{equation} \label{eqn: ARE solution bound}
    \norm{\bX}_{1} \leq \frac{M^2}{2\alpha} \norm{\bQ}_{1},
\end{equation}
where $M,\alpha \in \mathbb R_+$ are the constants associated with the stability bounds for $\bS(t) \in \sL(H)$ given in \eqref{eqn: exponential stability bound}.
\end{theorem}
\begin{proof}
Let $\bX \in \sL^s(H)$ be the defined by \eqref{eqn: Bochner ARE} and let $\mathcal D(\bA\bX+\bX\bA^*)$ be the domain of $\bA\bX + \bX\bA^*$, i.e.
$$
    \mathcal D\p{\bA\bX + \bX\bA^*} :=
    \Br{\phi\in H : \br{\bA\bX + \bX\bA^*}\phi \in H}.
$$
We will demonstrate that $\mathcal D(\bA\bX + \bX\bA^*) = H$. Using the definition of the infinitesimal generators $\bA$ and $\bA^*$ given in \eqref{eqn: A definition} and \eqref{eqn: A* def} respectively and the definition of $\bX\in\sL^s(H)$ given by \eqref{eqn: Bochner ARE}, we have that 
$$
\begin{aligned}
    &\p{\bA \bX + \bX\bA^*}\zeta\\
    &\quad= \int_0^{+\infty}\br{\bA\bS(t)\p{\bQ - \bX\bG\bX}\bS^*(t)
    + 
    \bS(t)\p{\bQ - \bX\bG\bX}\bS^*(t)\bA^*
    } \zeta dt \\
    &\quad= 
    \lim_{h\rightarrow 0^+} h^{-1}
    \int_0^{+\infty}\br{\p{\bS(t+h) - \bS(t)}\p{\bQ - \bX \bG \bX}\bS(t+h) + \bS(t)\p{\bQ-\bX\bG\bX}\p{\bS^*(t+h) - \bS^*(t)}}\zeta dt \\
    &\quad= \lim_{h\rightarrow 0^+} h^{-1}
    \int_0^{+\infty} \br{\bS(t+h)\p{\bQ - \bX\bG\bX}\bS(t+h) - \bS(t)\p{\bQ - \bX\bG\bX}\bS^*(t)}\zeta dt
\end{aligned}
$$
for all $\zeta \in \mathcal D(\bA\bX + \bX\bA^*)$ . Notice here that we may already take $\zeta$ to be in $H$ without consequence since $\bS(t)$ and $\bS^*(t)$ are bounded operators on $H$. Continuing, we have that
$$
\begin{aligned}
    &\lim_{h\rightarrow 0^+} h^{-1}
    \int_0^{+\infty} \br{\bS(t+h)\p{\bQ - \bX\bG\bX}\bS^*(t+h) - \bS(t)\p{\bQ - \bX\bG\bX}\bS^*(t)}\zeta dt \\
    &\quad =
    \lim_{h\rightarrow 0^+} h^{-1} \Br{\lim_{\tau\rightarrow +\infty} 
    \int_0^\tau \br{\bS(t+h)\p{\bQ - \bX\bG\bX}\bS^*(t+h) - \bS(t)\p{\bQ - \bX\bG\bX}\bS^*(t)}\zeta dt} \\
    &\quad =
        \lim_{h\rightarrow 0^+} h^{-1} 
        \Br{\lim_{\tau\rightarrow +\infty} 
        \br{
    \int_h^{\tau+h} \bS(t)\p{\bQ - \bX\bG\bX}\bS^*(t)\zeta dt - \int_0^\tau\bS(t)\p{\bQ - \bX\bG\bX}\bS^*(t)\zeta dt}} \\
    &\quad =
        \lim_{h\rightarrow 0^+} h^{-1} 
        \br{\lim_{\tau\rightarrow +\infty} 
    \int_\tau^{\tau + h} \bS(t)\p{\bQ - \bX\bG\bX}\bS^*(t)\zeta dt -
    \int_0^h\bS(t)\p{\bQ - \bX\bG\bX}\bS^*(t)\zeta dt} \\    
    &= \lim_{\tau\rightarrow +\infty} \br{\lim_{h\rightarrow 0^+} h^{-1}\int_\tau^{\tau+h} \bS(t)\p{\bQ - \bX\bG\bX}\bS^*(t)\zeta dt} - \bQ\zeta + \bX\bG\bX\zeta \\
    &= \lim_{\tau\rightarrow +\infty} \br{\bS(\tau)\p{\bQ - \bX\bG\bX}\bS^*(\tau)\zeta} - \bQ\zeta + \bX\bG\bX\zeta \\
    &=  - \bQ\zeta + \bX\bG\bX\zeta
\end{aligned}
$$
for all $\zeta \in H$ after applying the fact that $\bS(t) \in \sL(H)$, and consequently also $\bS^*(t) \in \sL(H)$, vanish in the $t\rightarrow +\infty$ limit as a consequence of their exponential stability. The interchange of limits used in the above sequence of equalities is allowable owing to the continuity of the function that the limiting operations are applied to. We have therefore demonstrated that
\begin{equation}
 \label{eqn: inspection}
    \p{\bA\bX + \bX\bA^*}\zeta = \p{-\bQ + \bX\bG\bX}\zeta
\end{equation}
for all $\zeta \in H$. This then implies that $\bX \in \sL^s(H)$ defined as a solution to \eqref{eqn: Bochner ARE} also necessarily satisfies \eqref{eqn: strong ARE}. 

Next, we demonstrate that the solution of the strong operator-valued Riccati equation \eqref{eqn: strong ARE} satisfies \eqref{eqn: Bochner ARE}. To do this, we derive the following weak form of the operator-value Riccati equation from \eqref{eqn: strong ARE}
\begin{equation} \label{eqn: weak form}
    \p{\br{\ARE} \phi, \psi}_H = 0,
\end{equation}
for all $\phi, \psi \in \mathcal D(\bA^*)$. It then follows from \cite[Proposition 4]{cheung2025approximation} that \eqref{eqn: weak form} is equivalent to the Bochner integral form of the operator-valued Riccati equation \eqref{eqn: Bochner ARE}. Therefore, a solution $\bX \in \sL^s(H)$ satisfying \eqref{eqn: Bochner ARE} also satisfies \eqref{eqn: strong ARE}. The uniqueness of $\bX\in\sL^s(H)$ that satisfies \eqref{eqn: strong ARE} is a consequence of the well-posedness of \eqref{eqn: Bochner ARE} determined in \cite{burns2015solutions}. This argument presented in previous paragraph along with this paragraph indicates that there exists only one positive semi-definite solution to \eqref{eqn: strong ARE} and the solution coincides with the solution of \eqref{eqn: Bochner ARE}. Symmetry of $\bX$ is easily determined through inspection taking the adjoint of both sides of \eqref{eqn: strong ARE}. 

We conclude the proof by deriving the solution bound \eqref{eqn: ARE solution bound} under the assumption that $\bQ \in \sJ_1^s(H)$. To that end, we consider the following weak integral form of the operator-valued Riccati equation.
\begin{equation} \label{eqn: weak integral form}
    \p{\psi, \bX\phi}_H = \int_0^{+\infty}  \p{\psi, \bS(t)\p{\bQ - \bX\bG\bX} \bS^*(t)\phi}_H dt 
\end{equation}
for all $\phi,\psi \in H$. It is clear that the solution $\bX \in \sL^s(H) $ to \eqref{eqn: Bochner ARE} also satisfies \eqref{eqn: weak integral form}. By choosing $\phi = \psi = e_i$, where $\Br{e_i}_{i=1}^\infty$ forms an orthonormal basis of $H$, we then have that 
$$
    \p{e_i, \bX e_i}_H + \int_0^{+\infty} \p{e_i, \bS(t)\bX \bG \bX \bS^*(t), e_i}_H dt = \int_0^{+\infty} \p{e_i,\bS(t)\bQ\bS^*(t)e_i}_H dt.
$$
Since $\bX\in \sL^s(H)$ and $\bG\in \sL^s(H)$ are symmetric positive semi-definite, we have that both terms in the left hand side of the above equation are nonnegative. It then follows that
$$
    \p{e_i, \bS(t) \bX\bG\bX \bS^*(t)e_i}_H = \norm{\bG^{\frac12}\bX\bS^*(t) e_i}_H^2 \geq 0,
$$
where $\bG^{\frac12} \in \sL(H)$ denotes the operator square root \cite[Theorem 7.38]{axler2024linear} of $\bG \in \sL^s(H)$. From this, it follows that
$$
    \p{e_i, \bX e_i}_H \leq \int_0^{+\infty} \p{e_i,\bS(t)\bQ\bS^*(t)e_i}_Hds,
$$
where summing both sides over all $i\in\mathbb N$ yields
$$
\begin{aligned}
    \trace{\bX} &\leq \int_0^{+\infty} \trace{\bS(t)\bQ\bS^*(t)}dt \\ 
    &\leq \int_0^{+\infty} \norm{\bS^*(t)}^2_{\sL\p{H}}\norm{\bQ}_{1}ds \\
    &\leq M^2\norm{\bQ}_{1} \int_0^{+\infty}  e^{-2\alpha t} dt \\
    &= \frac{M^2}{2\alpha} \norm{\bQ}_{\sJ_1\p{H}},
\end{aligned}
$$
and hence $\bX \in \sJ_1^s(H)$ and \eqref{eqn: ARE solution bound} follows from seeing that
$\trace{\bX} = \norm{\bX}_{1}$ because $\bX \in \sJ_1^s(H)$ is symmetric positive semi-definite. Since $\bX\in \sJ_1^s(H)$, we have that $\bA\bX+\bX\bA^*\in \sJ_1^s(H)$.  
\end{proof}

We proceed with a discussion on the strong form of the Sylvester equation in the following subsection. 

\subsection{Sylvester's Equation}
In the analysis presented in the following subsections, we frequently encounter the strong form of the operator-valued Sylvester equation (it is ``strong'' in the same sense \eqref{eqn: strong ARE} is the strong form of the operator-valued Riccati equation), given by 
\begin{equation} \label{eqn: Sylvester Equation}
    \bA_1 \bT + \bT \bA_2^* = \mathbf P
\end{equation}
where $\bA_1:= \bA - \mathbf K_1$ and $\bA_2: \bA - \mathbf K_2$ defined as generators of exponentially stable perturbation $C_0$-semigroups $\bS_1(t) \in \sL\p{H}$ and $\bS_2(t) \in \sL\p{H}$ respectively for all $t \in \mathbb R_+$  with $\mathbf K_1, \mathbf K_2 \in \sL\p{H}$ being positive semi-definite operators. Because of the way $\bA_1, \bA_2$ are defined, we have that their domains coincide with $\mathcal D(\bA)$, and hence, \eqref{eqn: Sylvester Equation} is well-defined. With this, we determine the following.
\begin{lemma} \label{lemma: Sylvester Equation}
    Assume that $\bA_1, \bA_2: \mathcal D(\bA) \rightarrow H$ are generators of exponentially stable $C_0$-semigroups $\bS_1(t) \in \sL\p{H}$ and $\bS_2(t) \in \sL\p{H}$ for all $t \in \mathbb R_+$, then the unique solution $\mathbf T \in \sL\p{H}$ of \eqref{eqn: Sylvester Equation} is given by the following Bochner integral representation
    \begin{equation} \label{eqn: sylvester solution}
        \bT = - \int_0^{+\infty} \bS_1(t)\mathbf P \bS_2^*(t)dt,
    \end{equation}
    where we have assumed that $\mathbf P \in \sL\p{H}$.
\end{lemma}
\begin{proof}
    We begin by first demonstrating that the integral in \eqref{eqn: sylvester solution} is well-defined. This is done by showing that the norm of the integrand is integrable over all of $\mathbb R_+$ \cite[Section 2, Theorem 2]{diestel1974vector}. We verify this claim in the following.
    $$
    \begin{aligned}
        \int_0^{+\infty} \norm{\bS_1(t) \mathbf P \bS_2^*(t)}_{\sL\p{H}}dt 
        &\leq \int_0^{+\infty} M_1 M_2 e^{-\alpha_1 t} e^{-\alpha_2 t} \norm{\mathbf P}_{\sL\p{H}} dt \\
        &\leq M_*^2 \norm{\mathbf P}_{\sL\p{H}}  \int_0^{+\infty} e^{-2\alpha_* t} dt \\
        &\leq \frac{M_*^2 }{2\alpha_*} \norm{\mathbf P}_{\sL\p{H}},
    \end{aligned}
    $$
    where $M_1, M_2, \alpha_1, \alpha_2 \in \mathbb R_+$ are the stability constants associated with the exponentially stable $C_0$-semigroups $\bS_1(t), \bS_2(t) \in \sL\p{H}$ respectively, $M_* := \max\Br{M_1, M_2}$ and $\alpha_* := \min \Br{\alpha_1, \alpha_2}$. Utilizing \eqref{eqn: sylvester solution} in \eqref{eqn: Sylvester Equation} and following a similar derivation as in the first paragraph of the proof of Theorem \ref{theorem: ARE well-posedness} demonstrates that $\mathbf T$ defined by \eqref{eqn: sylvester solution} is a solution of the strong form of Sylvester's equation. Uniqueness follows as a consequence of the linearity of the equation.  
\end{proof}

We now move on to discuss the dual problem to \eqref{eqn: strong ARE} that arises in the first-order optimality system of the penalized optimization problems presented in \S\ref{sec: problem 1} and \S \ref{sec: problem 2}.

\subsection{Dual Problem}
The dual problem arises in the derivation of the first-order optimality system by determining the Fr\'echet derivative of the Lagrangian saddle-point functional with respect to the primal variable $\bX\in\sJ_1^s(H)$. We will go through its derivation in \S\ref{subsec: derivation of optimality system}. For now, we will simply state the strong form of the dual problem and determine the its well-posedness. 

The dual problem to \eqref{eqn: strong ARE} is stated as follows: Seek a $\bLambda\in\sL^s(H)$ that satisfies
\begin{equation} \label{eqn: dual problem strong form}
    \p{\bA^* - \bG\bX}\bLambda + \bLambda\p{\bA-\bX\bG} = -\bW,
\end{equation}
where $\bA$ and $\bA^*$ are the generators of the exponentially stable $C_0$-semigroups $\bS(t)$ and $\bS^*(t)$ respectively, and $\bG \in \sJ_1^s(H)$ and $\bW \in \sL^s(H)$ are symmetric positive semi-definite operators, and $\bX\in \sJ_1^s(H)$ is the solution to \eqref{eqn: strong ARE}. The well-posedness of \eqref{eqn: dual problem strong form} is determined in the following.

\begin{lemma} \label{lemma: Lambda bound}
    Let $\bLambda \in \sL^s(H)$ be the solution to \eqref{eqn: dual problem strong form}, then there exist positive constants $M, \alpha \in \mathbb R_+$ satisfying
    $$
    \norm{\bLambda}_{\sL\p{H}} \leq \frac{M^2}{2\alpha}\norm{\bW}_{\sL(H)}.
    $$
    Furthermore, the solution $\bLambda \in \sL^s(H)$ is symmetric positive semi-definite.
\end{lemma}
\begin{proof}
    Let $\bT(t) \in \sL\p{H}$ be the exponentially stable $C_0$-semigroup generated by $\bA^* - \bG \bX$. Because \eqref{eqn: dual problem strong form} is a Sylvester equation,  we have from Lemma \ref{lemma: Sylvester Equation} that $\bLambda \in \sL^s(H)$ can be represented in the following Bochner integral form
    \begin{equation} \label{eqn: adjoint representation}
        \bLambda = \int_0^{+\infty} \bT(t)\bW\bT^*(t)dt.
    \end{equation}
    Because $-\bG\bX \in \sL\p{H}$ is a stabilizing perturbation to $\bA: \mathcal D(\bA) \rightarrow H$, we have that $\norm{\bT(t)}_{\sL\p{H}} \leq Me^{-\alpha t}$. It then follows that
    $$
        \norm{\bLambda}_{\sL\p{H}} \leq  M^2
        \norm{\bW}_{\sL(H)}
        \int_0^{+\infty} e^{-2\alpha t}dt = \frac{M^2}{2\alpha} \norm{\bW}_{\sL(H)},
    $$
    from which the bound presented in the lemma is proven. The symmetric positive semi-definite nature of $\bLambda \in \sL^s(H)$ comes from inspecting \eqref{eqn: adjoint representation} where taking the adjoint of both sides of the equation immediately verifies this claim. 
\end{proof}

\subsection{Parametrized Control Device Operator} \label{subsec: G assumptions}
In many application problems, the operator $\bG \in \sL^s(H)$ is parametertrized by a set of parameters, i.e., it is a function that maps a parameter space $\sP$ into the operator space $\sL^s(H)$. This parameter space corresponds to, e.g. control device placement locations \cite{5482053} and geometric design parameters \cite{edalatzadeh2019optimal}. We formalize the definition of the parametrized operators $\bG_p$ for $p\in\sP$ in the following. 

Let $\sP$ be a complex Hilbert space with its norm $\norm{\cdot}_{\sP}$ be induced by the inner product $\p{\cdot, \cdot}_{\sP}$ and $\bG_p \in \sJ_1^s(H)$ be a trace-class symmetric positive semi-definite operator parametrized by a parameter $p \in \sP$ so that the mapping $p \mapsto \bG_p$ is twice Fr\'echet differentiable with respect to $p\in \sP$ and that its first derivative $\frac{\partial \bG_{p}}{\partial p}\p{\cdot}$ is bounded as an operator in $\sL\p{\sP; \sJ_1^s(H)}$  and the second derivative $\frac{\partial^2\bG_p}{\partial p^2}(\cdot,\cdot)$ is bounded in $\sL\p{\sP;\sL\p{\sP;\sJ_1^s(H)}}$. The first assumption we make on $\bG_{\p{\cdot}}$ is that it is uniformly bounded with respect to any $p\in\sP$, i.e. there exists a positive constant $g\in \mathbb R_+$ that satisfies
\begin{equation} \label{eqn: G bound}
    \norm{\bG_p}_1 \leq g
\end{equation}
for all $p \in \sP$. Because we have assumed that $\bG_p \in \sJ_1^s(H)$ is Fr\'echet differentiable for all $p \in \sP$, it follows that it is Lipschitz continuous, i.e. there exists a positive constant $L_\bG \in \mathbb R_+$ so that
\begin{equation} \label{eqn: G lipschitz bound}
    \norm{\bG_{p_1} - \bG_{p_2}}_{1} \leq L_{\bG} 
    \norm{p_1 - p_2}_{\sP}
\end{equation}
for all $p_1,p_2 \in \sP$. To reduce notational clutter, we will denote $\bd\bG_p\p{\cdot} := \frac{\partial \bG_{p}}{\partial p}(\cdot)$. We will further assume that $\bd\bG_p$ is Lipschitz continuous, i.e. there exists a positive constant $L_{\bd\bG} \in \mathbb R_+$ that satisfies 
\begin{equation} \label{eqn: dG lipshitz bound}
\norm{\bd\bG_{p_1} - \bd\bG_{p_2}}_{\sL\p{\sP; \sJ_1^s(H)}} \leq L_{\bd\bG} \norm{p_1 - p_2}_{\sP}
\end{equation}
for all $p_1,p_2 \in \sP$. We will further assume that 
\begin{equation} \label{eqn: nonzero assumption}
    \bd\bG_p \neq 0
\end{equation}
for any $p\in \sP$ and that 
\begin{equation} \label{eqn: invertability assumption}
    \br{\bd\bG_p^*\bd\bG_p}^{-1} \in \sL\p{\sP}.
\end{equation}
Finally, we the twice differentiability assumption implies that
\begin{equation} \label{eqn: bounded second variation}
    \bd^2\bG_p(q,r) < \infty
\end{equation}
for all $q,r \in \sP$, where we have denoted $\bd^2\bG_p(\cdot,\cdot):=\frac{\partial^2 \bG_p}{\partial p^2}(\cdot,\cdot)$. The satisfaction of this assumption is one of the necessary conditions for the second variation of Lagrangian functionals studied in this work to be bounded.  

Under the parametrization of $\bG_p$ with respect to $p\in\sP$, we have that $\p{\bX,\bLambda} \in\sJ_1^s(H)\times\sL^s(H)$ is the solution to the following coupled equations
\begin{subequations}
    \begin{equation} \label{eqn: primal parametrized}
        \bA\bX+\bX\bA^* - \bX\bG_p\bX+\bQ = \mathbf 0
    \end{equation}
    \begin{equation} \label{eqn: dual parametrized}
        \p{\bA^*-\bG_p\bX}\bLambda + \bLambda\p{\bA-\bX\bG_p} = -\bW,
    \end{equation}
\end{subequations}
where $\bX=\bX(p)$ and $\bLambda = \bLambda(p)$. We determine in the following that both $\bX \in \sJ_1^s(H)$ and $\bLambda \in \sL^s(H)$ are Lipschitz continuous functions of $p \in \sP$.

\subsection{Lipschitz Continuity of the Primal and Dual Solutions}
In each of the first-order optimality systems associated with their penalized constrained optimization problems lies a fixed-point equation that must be satisfied. We will determine that the primal problem \eqref{eqn: primal parametrized} and the dual problem \eqref{eqn: dual parametrized} are Lipschitz continuous functions of $p\in \sP$ if $\bG_{\p{\cdot}} : \sP \rightarrow \sJ_1^s(H)$ satisfies the assumptions prescribed in \S\ref{subsec: G assumptions}. These Lipschitz continuity bounds will then be used to determine that each fixed-point equation has only one fixed point. A more detailed discussion of these fixed point problems will be presented in \S\ref{sec: problem 1} and \S\ref{sec: problem 2} respectively. For now, we focus exclusively on proving that $\bX\in\sJ_1^s(H)$ and $\bLambda \in \sL^s(H)$ are Lipschitz continuous with respect to $p \in \sP$. 

Throughout this section, we will denote $p_1,p_2\in\sP$ as any two arbitrary parameters and $\bX_1, \bX_2 \in \sJ_1^s(H)$ to be the solutions to
\begin{equation} \label{eqn: primal indexed}
    \bA\bX_i + \bX_i\bA^* - \bX_i\bG_{p_i}\bX_i + \bQ = \mathbf 0
\end{equation}
for $i=1,2$. Likewise, we denote $\bLambda_1, \bLambda_2 \in \sL^s(H)$ to be the solutions to 
\begin{equation} \label{eqn: dual indexed}
    \p{\bA^* - \bG_{p_i}\bX_i}\bLambda_i
    + 
    \bLambda_i\p{\bA-\bX_i\bG_{p_i}} = -\bW
\end{equation}
for $i=1,2$. We begin our analysis by determining that $\bX(p)$ is a Lipschitz continuous function of $p \in \sP$ in the following. 

\begin{lemma} \label{lemma: continuity wrt G}
    Assume that $\bA: \mathcal D(\bA) \rightarrow H$ is the generator of an exponentially stable $C_0$-semigroup and $\bQ \in \sJ_1^s(H)$ is a symmetric positive semi-definite operator. Further assume that $\bG_{\p{\cdot}} : \sP \rightarrow \sJ_1^s(H)$ is Lipshitz continuous on $\sP$ satisfying \eqref{eqn: G lipschitz bound}. Then the solution $\bX\in \sJ_1^s(H)$ to \eqref{eqn: primal indexed} is a Lipschitz continuous function of $p\in \sP$. Furthermore, there exists positive constants $L_{\bG}, M,\alpha \in \mathbb R_+$ so that
    $$
        \norm{\bX_1 - \bX_2}_{1} \leq \frac{L_{\bG}M^6}{8\alpha^3} \norm{\bQ}^2_{1} \norm{p_1 - p_2}_{\sP},
    $$
    where we have denoted $\bX_1, \bX_2 \in \sJ_1^s(H)$ to be the solution to \eqref{eqn: primal indexed} with the coefficient operators $\bG_{p_1}, \bG_{p_2} \in \sJ_1^s(H)$ determined by $p_1, p_2 \in \sP$ respectively.
\end{lemma}
\begin{proof}
    We begin by taking the difference between the equations for $\bX_1 \in \sJ_1^s\p{H}$ and $\bX_2 \in \sJ_1^s(H)$ respectively (see \eqref{eqn: primal indexed}), We have that the difference $\bX_1 - \bX_2 \in \sJ_1^s\p{H}$ is then the solution to the following Sylvester equation
    $$
        \br{\bA - \bX_1 \bG_{p_1}}\p{\bX_1 - \bX_2} 
        + \p{\bX_1 - \bX_2} \br{\bA^* - \bG_{p_2}\bX_2} = \bX_1\p{\bG_{p_1} - \bG_{p_2}} \bX_2.
    $$
    It then follows from Lemma \ref{lemma: Sylvester Equation} that the solution to the above equation satisfies the following integral equation
    \begin{equation} \label{eqn: A}
        \p{\bX_1 - \bX_2} = -\int_0^{+\infty} \bT_1(t)\bX_1\p{\bG_{p_1} - \bG_{p_2}}\bX_2\bT_2(t)dt, 
    \end{equation}
    where $\bT_1(t), \bT_2(t) \in \sL\p{H}$ are the $C_0$-semigroups generated by $\br{\bA - \bX_1\bG_{p_1}} : \mathcal D\p{\bA} \rightarrow H$ and $\br{\bA^* - \bG_{p_2}\bX_2} : \mathcal D\p{\bA^*} \rightarrow H$ respectively. 
    
    Since $\bA$ is the generator of an exponentially stable $C_0$-semigroup and that $\bX_1\bG_{p_1} \in \sJ_1^s(H)$ and $\bG_{p_2}\bX_2 \in \sJ_1^s\p{H}$ are bounded nonnegative operators, we have that $\br{\bA - \bX_1\bG_{p_1}}: \mathcal D(\bA) \rightarrow H$ and $\br{\bA^*-\bG_{p_2}\bX_2} : \mathcal D(\bA^*) \rightarrow H$ are also generators of exponentially stable $C_0$-semigroups. Furthermore, it follows that
    \begin{equation} \label{eqn: B}
        \norm{\bT_1(t)}_{\sL\p{H}} \leq Me^{-\alpha t} \textrm{ and } \norm{\bT_2(t)}_{\sL\p{H}} \leq Me^{-\alpha t}
    \end{equation}
    for all $t \in \mathbb R_+$, where $M,\alpha$ are the same constants associated with the unperturbed semigroup $\bS(t) \in \sL(H)$. 

    Norming both sides of \eqref{eqn: A} with respect to the $\sJ_1\p{H}$ norm then yields
    $$
        \norm{\bX_1 - \bX_2}_{1} \leq  \int_0^{+\infty} \norm{\bT_1(t)\bX_1\p{\bG_{p_1} - \bG_{p_2}}\bX_2\bT_2(t)}_{1} dt
    $$
    after applying the definition of the operator trace. It then follows that
    $$
        \norm{\bX_1 - \bX_2}_{1} \leq \frac{M^2}{2\alpha} \norm{\bX_1}_{1} \norm{\bX_2}_{1}\norm{\bG_{p_1} - \bG_{p_2}}_{\sL(H)}
    $$
    after applying the fact that $\sJ_1(H)$ is a two-sided *-ideal in $\sL(H)$ and the bounds provided in \eqref{eqn: B}. Applying \eqref{eqn: ARE solution bound} in the statement of Theorem \ref{theorem: ARE well-posedness} then yields
    $$
    \begin{aligned}
        \norm{\bX_1 - \bX_2}_{1} 
        &\leq \frac{M^6}{8\alpha^3} \norm{\bQ}^2_{1} \norm{\bG_{p_1} - \bG_{p_2}}_{\sL\p{H}} \\
        &\leq \frac{L_{\bG}M^6}{8\alpha^3} \norm{\bQ}^2_{1} \norm{p_1 - p_2}_{\sP},
    \end{aligned}
    $$    
    after applying the Lipschitz continuity assumption on $\bG_{\p{\cdot}}: \sP \rightarrow \sL^s(H)$.
\end{proof}

With Lemma \ref{lemma: Lambda bound}, we are now able to prove that $\bLambda \in \sL^s(H)$ is a Lipshitz continuous function of $p \in \sP$ in the following.
\begin{lemma} \label{lemma: Lambda Difference bound}
    Let $\bLambda \in \sL^s(H)$ satisfy \eqref{eqn: dual parametrized}. Then $\bLambda \in \sL^s(H)$ is a Lipschitz continuous function of $p \in \sP$ and there exists positive constants $M, \alpha \in \mathbb R_+$ so that
    $$
        \norm{\bLambda_{1} - \bLambda_{2}}_{\sL\p{H}} \leq \p{\frac{M^{10}g}{16\alpha^5}\norm{\bQ}^2_1 
        + \frac{M^6}{4\alpha^3}\norm{\bQ}_1 }L_{\bG} \norm{p_1 - p_2}_{\sP},
    $$
    for any $p_1, p_2 \in \sP$, where $g := \sup_{p\in\sP}\norm{\bG_p}_{\sL(H)}$ for any $p \in \sP$.
\end{lemma}
\begin{proof}
    We begin by taking the difference of the equations \eqref{eqn: dual indexed} between $i=1,2$. We arrive at the following Sylvester equation
    \begin{equation} \label{eqn: adjoint difference equation}
        \br{\bA^* - \bG_{p_2}\bX_2}\p{\bLambda_1 - \bLambda_2} + \p{\bLambda_1 - \bLambda_2}\br{\bA - \bX_1\bG_{p_1}} 
        = F(\bX_1, \bX_2, \bLambda_1, \bLambda_2, \bG_{p_1}, \bG_{p_2})
    \end{equation}
    where we have denoted
    $$
    \begin{aligned}
        &F(\bX_1, \bX_2, \bLambda_1, \bLambda_2, \bG_{p_1}, \bG_{p_2}) := \\
        & \qquad 
        \bG_{p_1}(\bX_1-\bX_2)\bLambda_1 
        + (\bG_{p_1} - \bG_{p_2})\bX_1\bLambda_1 
        + \bLambda_2 \bX_2(\bG_{p_1}-\bG_{p_2}) 
        + \bLambda_2(\bX_1 - \bX_2)\bG_{p_1}. 
    \end{aligned}
    $$
    Let now $\bT_2(t) \in \sL\p{H}$ be the exponentially stable $C_0$-semigroup generated by $\bA^* - \bG_{p_2}\bX_2$ and $\bT_1(t) \in \sL\p{H}$ be the exponentially stable $C_0$ semigroup generated by $\bA - \bX_1\bG_{p_1}$ respectively. With Lemma \ref{lemma: Sylvester Equation} we have that \eqref{eqn: adjoint difference equation} can be written in the following equivalent Bochner integral form
    $$
        (\bLambda_1 - \bLambda_2) = -\int_0^{+\infty} \bT_2(t) F\p{\bX_1, \bX_2, \bLambda_1, \bLambda_2, \bG_{p_1}, \bG_{p_2}} \bT_1(t) dt.
    $$
    Norming both sides with the $\sL\p{H}$ norm then allows us to see that 
    \begin{equation} \label{eqn: AAA}
    \begin{aligned}
        \norm{\bLambda_1 - \bLambda_2}_{\sL\p{H}} &\leq \int_0^{+\infty} M^2 e^{-2\alpha t } \norm{F(\bX_1, \bX_2, \bLambda_1, \bLambda_2, \bG_{p_1}, \bG_{p_2})}_{\sL\p{H}} dt \\
        &= \frac{M^2}{2\alpha} \norm{F(\bX_1, \bX_2, \bLambda_1, \bLambda_2, \bG_{p_1}, \bG_{p_2})}_{\sL\p{H}}.
    \end{aligned}
    \end{equation}
    
    We now bound $\norm{F\p{\bX_1,\bX_2,\bLambda_1, \bLambda_2, \bG_{p_1}, \bG_{p_2}}}_{\sL\p{H}}$. Recall \eqref{eqn: G bound}, where we have assumed $\norm{\bG_{p}}_{\sL(H)} = g$ for all $p \in \sP$. We have then
    $$
    \begin{aligned}
        &\norm{F(\bX_1, \bX_2, \bLambda_1, \bLambda_2, \bG_{p_1}, \bG_{p_2})}_{\sL\p{H}} \\
        & \quad \leq \norm{\bG_{p_1}(\bX_1-\bX_2)\bLambda_1}_{\sL\p{H}}
        + \norm{(\bG_{p_1} - \bG_{p_2})\bX_1\bLambda_1 }_{\sL\p{H}} 
        \\&\qquad
        + \norm{\bLambda_2 \bX_2(\bG_{p_1}-\bG_{p_2})}_{\sL\p{H}} 
        + \norm{\bLambda_2(\bX_1 - \bX_2)\bG_{p_1}}_{\sL\p{H}} \\ 
        &\quad \leq \p{\frac{M^8g}{8\alpha^4} \norm{\bQ}^2_1 + \frac{M^4}{2\alpha^2}\norm{\bQ}_1} \norm{\bG_{p_1} - \bG_{p_2}}_{\sL(H)},
    \end{aligned}
    $$
    after applying Theorem \ref{theorem: ARE well-posedness}, Lemmas \ref{lemma: continuity wrt G}, and \ref{lemma: Lambda bound}. Inserting this bound into \eqref{eqn: AAA} then results in 
    $$
        \norm{\bLambda_1 - \bLambda_2}_{\sL\p{H}} 
        \leq \p{\frac{M^{10}g}{16\alpha^5}\norm{\bQ}^2_1 
        + \frac{M^6}{4\alpha^3}\norm{\bQ}_1 } \norm{\bG_{p_1} - \bG_{p_2}}_{\sL(H)}.
    $$
    Applying the Lipschitz continuity assumption \eqref{eqn: G lipschitz bound} yields the result of this lemma. 
\end{proof}

\subsection{The Critical Cone}
The sufficient second-order optimality condition requires that the Hessian evaluated at the stationary point of the Lagrangian be positive definite in the directions in the critical cone associated with the constraint. Loosely speaking, the critical cone is a subset of the tangent space of the constraint manifold evaluated at $\p{\bX_{opt}, p_{opt}}$ that maps the first Gat\'eux (directional) derivative of the constraint to zero. We will utilize the second-order optimality condition to demonstrate that the solution to the first-order optimality system is indeed the unique minimizer to the associated penalized constrained optimization problems studied in this work. 

Let us define 
$$
c(\bX, p) := \bA\bX + \bX\bA^* - \bX\bG_p\bX+\bQ
$$
to be the constraint function associated with the operator-valued Riccati equation, where again $\bQ \in \sJ_1^s(H)$ and $\bG_{(\cdot)} : \sP \rightarrow \sJ_1^s(H)$ satisfies \eqref{eqn: G bound}, \eqref{eqn: G lipschitz bound}, and \eqref{eqn: dG lipshitz bound}. The critical cone for the constrained optimizer is defined by the following set
$$
    \mathcal K(\bX_{opt},p_{opt}) :=
    \Br{
        \p{\bPhi,q} \in \sJ_1^s(H) \times \sP : 
        \left.\frac{\partial c}{\partial \bX}\right|_{\p{\bX_{opt},p}_{opt}}(\bPhi) = 0 \textrm{ and } 
        \left.\frac{\partial c}{\partial p}\right|_{\p{\bX_{opt}, p_{opt}}}(q) = 0
    }.
$$
We now characterize $\mathcal K(\bX_{opt}, p_{opt})$. Taking the first Gat\'eux derivative of $c(\cdot, \cdot)$ with respect to $\bX \in \sJ_1^s(H)$ yields
$$
\begin{aligned}
    \frac{\partial c}{\partial \bX}(\bPhi) 
    &= \bA\bPhi + \bPhi \bA^* - \bPhi\bG_p\bX - \bX\bG_p\bPhi \\
    &=\p{\bA-\bX\bG_p}\bPhi + \bPhi\p{\bA^* - \bG_p \bX} \\
    &=0
\end{aligned}
$$
for all $\bPhi \in\sJ_1^s(H)$. Because $\bPhi \in \sJ_1^s(H)$ is now the solution of a Sylvester equation (see \eqref{eqn: Sylvester Equation}) with $\mathbf 0$ as the data, we have that $\bPhi \equiv \mathbf 0$. Continuing, we have that 
$$
    \frac{\partial c}{\partial p}(q) = -\bX\bd\bG_p(q) \bX = 0.
$$
In general, the set of $q \in \sP$ satisfying the above condition is nonempty. Combining the two observations made above, we have that $\mathcal K(\bX_{opt}, p_{opt})$ can be characterized by the following
\begin{equation} \label{eqn: critical cone}
    \mathcal K(\bX_{opt}, p_{opt}) := \Br{\p{\mathbf 0, q} \in \sJ_1^s(H)\times\sP : -\bX_{opt}\bd\bG_{p_{opt}}(q)\bX_{opt} = \mathbf 0}.
\end{equation}

With the preliminary analysis completed, we are now ready to present and analyze the two penalized constrained optimization problems of interest in the following two sections of this paper. 

\section{Control Penalized Constrained Trace Minimization}
\label{sec: problem 1}
\subsection{Problem Statement}
The control penalization technique \cite{troltzsch2010optimal} is a well-known and often applied method of regularizing the solution of optimization problems. It has the benefit of implicitly constraining the parameter space and improving the well-posedness properties of the optimization problem. We study this technique in the context of control device design and placement in this section. The ideas presented in this section will serve as a pedagogical stepping stone for the more complex arguments needed to analyze the problem presented in the following section. 

The control penalized optimization problem of interest is to seek a $\p{\bX_{opt},p_{opt}} \in \sJ_1^s(H) \times \sP$ that minimizes
\begin{equation} 
\label{eqn: cost functional 1}
    \mathcal J_\beta(\bX,p) := \trace{\bX \bW} + \frac{\beta}{2}\norm{p}_{\sP}^2
\end{equation}
constrained by the strong operator-valued Riccati equation
\begin{equation}
\label{eqn: operator-valued Riccati equation}
    \bA\bX + \bX\bA^* - \bX\bG_p\bX + \bQ = \mathbf 0,
\end{equation}
where $\bW \in \sL^s(H)$ is a symmetric positive semi-definite weighting operator, $\bA: \mathcal D(\bA) \subset H \rightarrow H$ is the generator of the exponentially stable $C_0$-semigroup $\bS(t) \in \sL(H)$, $\bQ \in \sJ_1^s(H)$ is a nonnegative operator, and $\bG_{\p{\cdot}}: \sP \rightarrow \sJ_1^s(H)$ is the parametrized operator associated with the control device. The definition of $\mathcal J_{\beta}(\cdot)$ presented in \eqref{eqn: cost functional 1} is equivalent to the cost functional presented in the introduction of this work. This is because for each $p \in \sP$, there is only one $\bX\in \sJ_1^s(H)$ that satisfies \eqref{eqn: operator-valued Riccati equation}.

In the following discussion, we will first present the first-order optimality system associated with the penalized optimization problem studied in this section, then we provide a derivation of this set of equations. Next we determine that there exists only one solution to the first order optimality system and that this solution satisfies the second-order sufficient conditions to qualify as a constrained minimizer of the cost functional \eqref{eqn: cost functional 1}. In other words, there exists only one global constrained minimizer for \eqref{eqn: cost functional 1}. 

\subsection{First-Order Optimality System}
The first-order optimality system associated with the constrained optimization problem is given as follows: Seek a $\p{\bX_{opt},\bLambda_{opt}, p_{opt}} \in \sJ_1^{s}(H) \times \sL^{s}(H) \times \sP$ that satisfies 

\begin{subequations} \label{eqn: optimality system}
    \noindent\textrm{\textbf{Primal Problem: }}
    \begin{equation}
    \label{eqn: primal problem}
        \bA\bX+\bX\bA^*-\bX\bG_p\bX+\bQ = 0
    \end{equation}
    \textrm{\textbf{Dual Problem: }}
    \begin{equation}
    \label{eqn: dual problem}
        \p{\bA^* - \bG_p\bX}\bLambda + \bLambda\p{\bA-\bX\bG_p} = -\bW
    \end{equation}
    \textrm{\textbf{Optimality Condition: }}
    \begin{equation}
    \label{eqn: optimality condition}
        p = \frac{1}{\beta}\bd\bG_p^* \bX\bLambda\bX.           
    \end{equation}
\end{subequations}
The variable $\bLambda \in \sL^s(H)$ is the dual solution that arises from applying the Lagrange multiplier to \eqref{eqn: primal problem}. We will determine that there is only one solution $\p{\bX_{opt}, \bLambda_{opt}, p_{opt}}\in \sJ_1^s(H)\times\sL^s(H)\times\sP$ that satisfies \eqref{eqn: optimality system} and that $\p{\bX_{opt}, p_{opt}}$ is in fact a constrained minimizer of \eqref{eqn: cost functional 1}. 

\subsubsection{Derivation of the Optimality System}
\label{subsec: derivation of optimality system}
We now derive the first-order optimality system \eqref{eqn: optimality system}. We begin by introducing the Lagrange multiplier $\bLambda \in \sL(H)$. The Lagrangian functional associated with the cost functional \eqref{eqn: cost functional 1} with constraint \eqref{eqn: operator-valued Riccati equation} is the following
\begin{equation}
\label{eqn: lagrangian functional 1}
\mathcal L\p{\bX,p,\bLambda} := 
\dual{\bI, \bX\bW} +
\frac{\beta}{2} \norm{p}^2_{\sP} + 
\dual{\bLambda,\br{\bA\bX + \bX\bA^* - \bX\bG_p \bX + \bQ}},
\end{equation}
where we have utilized the identity $\trace{\bX\bW} := \dual{\bI, \bX\bW}$. Because $\dual{\bLambda,\cdot}$ is a functional belonging to $\sJ_1(H)'$, we have that \eqref{eqn: lagrangian functional 1} is well-defined since we have demonstrated that there exists a $\bX \in \sJ_1^s(H)$ that satisfies \eqref{eqn: primal problem} in the $\sJ_1(H)$ topology in Theorem \ref{theorem: ARE well-posedness}. 

The first order necessary condition for optimality, i.e. that $\p{\bX,p,\bLambda} \in \sJ_1(H) \times \sP \times \sL(H)$ is a saddle-point of \eqref{eqn: lagrangian functional 1},  is the following
\begin{equation}
\label{eqn: weak first order conditions 1}
    \frac{\partial \mathcal L}{\partial \bX}(\bPsi) = 0, 
    \quad 
    \frac{\partial \mathcal L}{\partial p}(q) = 0, \quad
    \frac{\partial \mathcal L}{\partial \bLambda}(\bPhi) = 0,
\end{equation}
for all $\p{\bPsi, q, \bPhi} \in \sL(H)\times \sP \times \sJ_1(H)$. In this work, we choose to work with the strong form of \eqref{eqn: weak first order conditions 1}, i.e.
\begin{equation}
\label{eqn: strong first order conditions 1}
    \frac{\partial \mathcal L}{\partial \bX} = 0, 
    \quad 
    \frac{\partial \mathcal L}{\partial p} = 0, 
    \quad
    \frac{\partial \mathcal L}{\partial \bLambda} = 0,
\end{equation}
because we have already derived the necessary theoretical results for the well-posedness analysis using the strong form of the equation that arise from \eqref{eqn: strong first order conditions 1}. It is easily determined that \eqref{eqn: strong first order conditions 1} being satisfied also implies that \eqref{eqn: weak first order conditions 1} is also satisfied, therefore it is sufficient to consider \eqref{eqn: strong first order conditions 1} in our analysis. Furthermore, \eqref{eqn: weak first order conditions 1} implies \eqref{eqn: strong first order conditions 1} because $\sJ_1(H)$ and $\sL(H)$ form a duality pairing under the trace operator. We proceed in deriving each term in \eqref{eqn: strong first order conditions 1} in the remaining paragraphs of this subsection.

We begin our discussion by deriving \eqref{eqn: primal problem}. Taking the Gat\'eux derivative of $\mathcal L(\cdot, \cdot,\cdot)$ with respect to $\bLambda$ yields
$$
    \frac{\partial \mathcal L}{\partial \bLambda}(\bPhi) = 
    \dual{\bPhi,\br{\bA\bX + \bX\bA^* - \bX\bG_p \bX + \bQ}}
$$
for all $\bPhi \in \sL(H)$. Setting $\frac{\partial \mathcal L}{\partial \bLambda}(\bPhi) = 0$ then yields the following necessary condition
\begin{equation}
    \label{eqn: weak primal equation}
    \dual{\bPhi,\br{\bA\bX + \bX\bA^* - \bX\bG_p \bX + \bQ}} = 0
\end{equation}
for all $\bPhi \in \sL(H)$. We have that \eqref{eqn: weak primal equation} is true if and only if \eqref{eqn: primal problem} is satisfied  because there is a one-to-one correspondence between  $\dual{\bPhi,\cdot}$ and  $\sJ_1(H)'$, and observing that Theorem \ref{theorem: ARE well-posedness} indicates that \eqref{eqn: strong ARE} is satisfied in the $\sJ_1(H)$ norm topology.

We now derive \eqref{eqn: dual problem}. The Gat\'eux derivative of $\mathcal L(\cdot, \cdot, \cdot)$ with respect to $\bX$ is the following
$$
\begin{aligned}
    \frac{\partial\mathcal L}{\partial \bX}(\bPsi) &=
    \dual{\bI, \bPsi\bW} +
    \dual{\bLambda, \br{ \p{\bA  - \bX \bG_p}\bPsi + \bPsi \p{ \bA^* - \bG_p\bX}}} \\
    &= \dual{\bW, \bPsi} 
        + \dual{\p{\bA^* - \bG_p \bX^*}\bLambda, \bPsi}
        + \dual{\bLambda^*, (\bA-\bX^*\bG_p)\bPsi^*} \\
    &= \dual{\bW, \bPsi} 
        + \dual{\p{\bA^* - \bG_p \bX^*}\bLambda, \bPsi}
        + \dual{(\bA^*-\bG_p\bX)\bLambda^*, \bPsi^*} \\
   &= \dual{\bW, \bPsi} 
        + \dual{\p{\bA^* - \bG_p \bX^*}\bLambda, \bPsi}
        + \dual{\bLambda \p{\bA-\bX^*\bG_p}, \bPsi} \\
   &= \dual{\bW, \bPsi} 
        + \dual{\p{\bA^* - \bG_p \bX}\bLambda, \bPsi}
        + \dual{\bLambda \p{\bA-\bX\bG_p}, \bPsi}
\end{aligned}
$$
for all $\bPsi \in \sJ_1(H)$. Note that we have made heavy use of \eqref{eqn: operator adjoint trace identity} in the derivation above. In the final equality, we have applied the property that $\bX^* = \bX$. Setting $\frac{\partial L}{\partial \bX}(\bPsi) = 0$ then results in 
\begin{equation}
\label{eqn: operator-weak dual}
\dual{\p{\bA^* - \bG_p \bX}\bLambda, \bPsi}
        + \dual{\bLambda \p{\bA-\bX\bG_p}, \bPsi}
    = -\dual{\bW, \bPsi}        
\end{equation}
for all $\bPsi \in \sJ_1(H)$. Because $\dual{\br{\p{\bA^*-\bG_p\bX}\bLambda + \bLambda\p{\bA - \bX\bG_p} + \bW},\cdot}$ is defined in the weak-* topology of $\sJ_1(H)'$, we have that \eqref{eqn: operator-weak dual} is satisfied if and only if \eqref{eqn: dual problem} is satisfied. This is because the zero element in $\sJ_1(H)'$ is the only functional that maps any element in $\sJ_1(H)$ to zero in the complex plane. 

We conclude this section with the derivation of \eqref{eqn: optimality condition}.
$$
\begin{aligned}
\frac{\partial L}{\partial p}(q) 
&= 
\beta\p{p,q}_{\sP}
-\dual{\bLambda, \bX \bd\bG_p(q)\bX} \\
&=\beta\dual{p,q}_\sP - \dual{\bX^*\bLambda, \bd\bG_p(q)\bX}\\
&=\beta\dual{p,q}_{\sP} - \dual{\bLambda^*\bX, \bX^*\bd\bG_p(q)^*} \\
&=\beta\dual{p,q}_{\sP} - \dual{\bX\bLambda^*\bX,\bd\bG_p(q)^*} \\ 
&=\beta\dual{p,q}_{\sP} - \dual{\bX^*\bLambda\bX^*,\bd\bG_p(q)} \\
&=\p{\beta p-\bd\bG_p^*\bX\bLambda\bX, q}_{\sP} \\
&= 0,
\end{aligned}
$$
after applying the fact that $\bX^* = \bX$, $\bLambda^* = \bLambda$, and making heavy use of \eqref{eqn: operator adjoint trace identity}. We then have that
\begin{equation} \label{eqn: weak optimality condition}
(p,q)_\sP = \frac{1}{\beta}\p{\bd\bG_p^*\bX\bLambda\bX, q}_\sP
\end{equation}
for all $q \in \sP$. We then arrive at \eqref{eqn: optimality condition} since \eqref{eqn: weak optimality condition} can only be satisfied if its strong form is satisfied since $\sP$ is a Hilbert space.

\subsubsection{Well-Posedness of the Optimality System}
\label{subsubsec: well-posedness 1}
We now determine the conditions that must be satisfied in order for \eqref{eqn: optimality system} to only have one solution. First, notice that \eqref{eqn: primal problem} is well-posed for any $p \in \sP$ as a consequence of Theorem \ref{theorem: ARE well-posedness} and observing that $\bG_p \in \sJ_1^s(H)$. Next, notice that $\bX \in \sJ_1^s(H)$ is coupled to \eqref{eqn: dual problem} as an input parameter. Lemma \ref{lemma: Lambda bound} indicates that \eqref{eqn: dual problem} is well-posed for any $\bX\in\sJ_!^s(H)$. Because \eqref{eqn: primal problem} is independent of $\bLambda \in \sL^s(H)$, there exists a unique $\p{\bX, \bLambda} \in \sJ_1^s(H)\times\sL^s(H)$ that satisfies the coupled equation \eqref{eqn: primal indexed} and \eqref{eqn: dual indexed} for any choice of $p \in \sP$. Because of this, we have determined the existence of a mapping $p \mapsto \p{\bX(p),\bLambda(p)}$ for any $p \in \sP$. This mapping is continuous as a consequence of Lemmas \ref{lemma: continuity wrt G} and \ref{lemma: Lambda Difference bound}. It is our goal then to determine the conditions by which we can select a unique $p \in \sP$ so that $\p{\bX(p),\bLambda(p),p} \in \sJ_1^s(H)\times \sL^s(H)\times\sP$ satisfies \eqref{eqn: optimality condition}. 

Let us define $f(p): \sP \rightarrow \sP$ as follows
$$
    f(p) := \bd\bG_p^*\bX(p)\bLambda(p)\bX(p),
$$
where $\p{\bX(p),\bLambda(p)}$ satisfies \eqref{eqn: primal problem} and \eqref{eqn: dual problem} respectively. It then becomes clear that \eqref{eqn: optimality condition} can be written in the following fixed point form
$$
    p = f(p).
$$
It is our goal to invoke the Banach fixed point theorem \cite[Theorem 3.7-1]{ciarlet2025linear} to determine the conditions under which the function $f(\cdot)$ is a contractive map, i.e., $f(\cdot)$ satisfies
$$
    \norm{f(p_1) - f(p_2)}_{\sP} \leq k \norm{p_1 - p_2}_{\sP}
$$
with $k < 1$ for any $p_1,p_2 \in \sP$. More concretely, we wish to demonstrate for any $p_1,p_2\in\sP$, that
\begin{equation} \label{eqn: contraction condition}
    \frac{1}{\beta}\norm{\bd\bG_{p_1}^*\bX_1\bLambda_1\bX_1 - \bd\bG_{p_2}^*\bX_2\bLambda_2\bX_2}_{\sP} \leq k\norm{p_1 - p_2}_{\sP}.
\end{equation}

We now characterize the constant $k\in \mathbb R_+$. Lemmas \ref{lemma: Lambda bound}, \ref{lemma: continuity wrt G}, and \ref{lemma: Lambda Difference bound} determines that both $\bX\in \sJ_1^s(H)$ and $\bLambda\in \sL^s(H)$ are Lipschitz continuous functions of $p \in \sP$ under the assumption that \eqref{eqn: G bound} and \eqref{eqn: G lipschitz bound} are satisfied. If we further assume \eqref{eqn: dG lipshitz bound} is satisfied, then it follows that
$$
\begin{aligned}
    \frac{1}{\beta}
    \norm{\bd\bG_{p_1}^*\bX_1 \bLambda_1 \bX_1 - \bd\bG_{p_2}^*\bX_2 \bLambda_2 \bX_2 }_{\sP} &\leq 
    \frac{1}{\beta}
    \bigg(
    \norm{\bd\bG_{p_1} - \bd\bG_{p_1}}_{\sL(\sP; \sJ_1(H))}
    \norm{\bX_1}^2_{1}\norm{\bLambda_2}_{\sL(H)} \\
    &\qquad + 
    \norm{\bd\bG_{p_1}}_{\sL\p{\sP; \sJ_1(H)}}
    \norm{\bX_1-\bX_2}_1 \norm{\bLambda_2}_{\sL(H)}
    \norm{\bX_2}_1 \\
    &\qquad +
    \norm{\bd\bG_{p_1}}_{\sL(\sP;\sJ_1(H))}
    \norm{\bX_1}_1\norm{\bLambda_1 - \bLambda_2}_{\sL(H)}\norm{\bX_2}_1\\
    &\qquad+ 
    \norm{\bd\bG_{p_1}}_{\sL\p{\sP; \sJ_1(H) }}
    \norm{\bX_1}_1
    \norm{\bLambda_1}_{\sL(H)}
    \norm{\bX_1 - \bX_2}_1
    \bigg) \\
    &\leq \frac{1}{\beta}
    \bigg(
        \frac{M^6}{16\alpha^3}
        \norm{\bQ}_1^2\norm{\bW}_{\sL(H)}
        \norm{\bd\bG_{p_1} - \bd\bG_{p_1}}_{\sL(\sP; \sJ_1(H))}\\
    &\qquad + \frac{C_{\bd\bG}M^4}{2\alpha^2}\norm{\bQ}_1 \norm{\bW}_{\sL(H)}\norm{\bX_1 - \bX_2}_1 \\
    &\qquad + \frac{C_{\bd\bG}M^4}{2\alpha^2}\norm{\bQ}_1^2 \norm{\bLambda_1 - \bLambda_2}_{\sL(H)}
        \bigg) \\
    &\leq\frac{1}{\beta}
    \bigg(
         \frac{L_{\bd\bG}M^6}{16\alpha^3}
        \norm{\bQ}_1^2\norm{\bW}_{\sL(H)}
        \norm{p_1 - p_2}_\sP\\
         &\qquad + \frac{L_{\bd\bG}C_{\bd\bG}M^{10}}{16\alpha^5}\norm{\bQ}_1^3 \norm{\bW}_{\sL(H)}\norm{p_1 - p_2}_\sP \\
         &\qquad + 
         \frac{L_{\bG}C_{\bd\bG}M^4}{2\alpha^2}\norm{\bQ}_1^2 
          \p{\frac{M^{10}\gamma}{16\alpha^5}\norm{\bQ}^2_1 
        + \frac{M^6}{4\alpha^3}\norm{\bQ}_1 }
        \norm{p_1 - p_2}_\sP\bigg). \\
        &= k_{\alpha, \beta, M, \bQ, \bW, \bG} \norm{p_1 - p_2}_{\sP}
\end{aligned}
$$
Therefore, we have demonstrated that \eqref{eqn: optimality condition} has a unique solution if the penalty parameter $\beta$ is chosen sufficiently large enough to sufficiently reduce the contributions of $\alpha, M, \norm{\bQ}_1,\norm{\bW}_{\sL\p{H}}$, and the constants associated with $\bG_{\p{\cdot}}$. 

We now analyze the Hessian operator $\bd^2 \mathcal L_{opt}\br{\cdot, \cdot} : \br{\sJ_1^s(H) \times \sP}\times\br{\sJ_1^s(H) \times \sP} \rightarrow \mathbb R$ associated with the Lagrangian cost functional \eqref{eqn: lagrangian functional 1} evaluated at $\p{\bX_{opt}, p_{opt}, \bLambda_{opt}}$. We hope to demonstrate that the second-order optimality condition is satisfied, i.e. that $\bd^2 \mathcal L_{opt}\br{(\bPhi,q),(\bPhi,q)}$ is positive definite for every $\p{\bPhi, q}$ in the critical cone $\mathcal K \p{\bX_{opt}, p_{opt}}$ (as defined in \eqref{eqn: critical cone}), We now define $\bd^2\mathcal L_{opt}\br{\cdot, \cdot}$ as follows
$$
\begin{aligned}
\bd^2\mathcal L_{opt}\br{(\bPhi,q),(\bPsi,r)} &=
    -\dual{\bLambda_{opt}, \br{\bPhi\bG_{p_{opt}}\bPsi + \bPsi\bG_{_{opt}}\bPhi}}\\
    &\qquad-\dual{\bLambda_{opt}, \br{\bPhi \bd\bG_{p_{opt}}(r)\bX_{opt}+\bX_{opt}\bd\bG_{p_{opt}}(r)\bPhi}} \\
    &\qquad-\dual{\bLambda_{opt}, \br{\bPsi \bd\bG_{p_{opt}}(q)\bX_{opt}+\bX_{opt}\bd\bG_{p_{opt}}(q)\bPsi}} \\
    &\qquad+ \beta\p{q,r}_{\sP} - \dual{ \bLambda_{opt}, \bX_{opt} \bd^2\bG_{p_{opt}}\p{q,r}\bX_{opt}} 
\end{aligned}
$$
for all $\p{\bPhi, q},\p{\Psi, r}\in\sJ_1(H)\times \sP$. Now restricting $\p{\bPhi, q}$ and $\p{\bPsi, r}$ to the critical cone $\mathcal K(\bX_{opt}, p_{opt})$ then yields
$$
\begin{aligned}
\bd^2\mathcal L_{opt}\br{(\bPhi,q),(\bPsi,r)} &=
   \beta\p{q,r}_{\sP} - \dual{ \bLambda_{opt}, \bX_{opt} \bd^2\bG_{p_{opt}}\p{q,r}\bX_{opt}} 
\end{aligned}
$$
since $\bPhi = \bPsi = 0$ must be satisfied for any element belonging to $\mathcal K\p{\bX_{opt}, p_{opt}}$. Taking $q = r$ then immediately indicates that $\bd^2 \mathcal L_{opt}$ is a positive semi-definite operator on $\mathcal K(\bX_{opt}, p_{opt})$ if $\beta$ is chosen sufficiently large, thereby satisfying the necessary second order condition for $\p{\bX_{opt}, p_{opt}}$ to qualify as a constrained minimizer to the cost functional \eqref{eqn: cost functional 1}. 

We formalize our findings in the following. 

\begin{theorem} \label{theorem: well-posedness problem 1}
Assume that $\bA: \mathcal D(\bA) \rightarrow H$ is the generator of a $C_0$-semigroup, $\bQ \in \sJ_1^s(H)$, and $\bW \in \sL^s(H)$. If conditions \eqref{eqn: G bound}, \eqref{eqn: G lipschitz bound}, and \eqref{eqn: dG lipshitz bound} are satisfied for the mapping $\bG_{(\cdot)}: \sP \rightarrow \sJ_1^s(H)$, then there exists a unique solution $\p{\bX_{opt}, \bLambda_{opt}, p} \in \sJ_1^s(H) \times \sL^s\p{H} \times \sP$ that satisfies the first-order optimality system \eqref{eqn: optimality system} associated with the constrained weighted trace minimization problem provided that the penalty parameter $\beta\in\mathbb R_+$ is chosen sufficiently large enough so that $k_{\alpha, \beta, M, \bQ, \bW, \bG} < 1$.

Further assume that \eqref{eqn: bounded second variation} is satisfied for the mapping $\bG_{(\cdot)}: \sP \rightarrow \sJ_1^s(H)$. Then $\p{\bX_{opt}, p_{opt}} \in \sJ_1^s(H)\times \sP$ is the unique constrained minimizer for the cost functional \eqref{eqn: cost functional 1} if $\beta \in \mathbb R_+$ is sufficiently large. 
\end{theorem}

In general, Theorem \ref{theorem: well-posedness problem 1} indicates that $\beta$ can be chosen arbitrarily large to force $\mathcal J_\beta(\cdot)$ to have a unique minimizer. 

\section{Approximate Control Constraint Enforcement}
\label{sec: problem 2}
In this section, we consider a penalized technique to approximately enforce a trace constraint on the operator $\bG_p \in \sJ_1^s(H)$ for all $p \in \sP$. One important property of this penalization approach is that the approximate constraint enforcement becomes exact as the penalization parameter reaches positive infinity. It is through this mechanism by which we are able to determine that a wide class of sensor placement and design problems admits a unique constrained minimizer. Much of the general techniques used to analyze the problem described in this section have been discussed in the previous section. Because of this, we will only briefly touch upon details that have direct analogs to what was discussed for the previous problem and focus our attention on the technicalities associated with the current problem of interest.

\subsection{Problem Statement}
The penalized optimization problem analyzed in this section is to seek a constrained minimizer $\p{\bX_{opt}, p_{opt}} \in \sJ_1^s(H) \times \sP$ for the following cost functional
\begin{equation}
\label{eqn: cost functional 2}
    \mathcal J_\beta (\bX, p) := \trace{\bX\bW} + \frac{\beta}{2} \br{\trace{\bG_p} - \gamma}^2
\end{equation}
subject to
\begin{equation}
    \bA\bX + \bX\bA^* - \bX\bG_p\bX+\bQ = \mathbf 0,
\end{equation}
where we assume that $\bQ \in \sJ_1^s(H)$ is again a positive semi-definite trace-class operator. 

\subsection{First-Order Optimality System}
The Lagrangian functional associated with the constrained optimization problem studied in this section is the following
\begin{equation} \label{eqn: Lagrangian functional 2}
    \mathcal L(\bX,p,\bLambda) := 
    \dual{\mathbf I, \bX\bW} 
    + \frac{\beta}{2}\br{\trace{\bG_p} - \gamma}^2
    + \dual{\bLambda, \bA \bX + \bX\bA^* - \bX\bG_p\bX + \bQ}.
\end{equation}
Taking the first Fr\'echet derivative of $\mathcal  L(\cdot,\cdot,\cdot)$ with respect to $\bX \in \sJ_1^s(H)$, $p \in \sP$, and $\bLambda \in \sL^s(H)$ yields the first-order optimality system associated with \eqref{eqn: Lagrangian functional 2} stated below. 

\begin{subequations} \label{eqn: optimality system 2}
    \noindent\textrm{\textbf{Primal Problem: }}
    \begin{equation}
    \label{eqn: primal problem 2}
        \bA\bX+\bX\bA^*-\bX\bG_p\bX+\bQ = 0
    \end{equation}
    \textrm{\textbf{Dual Problem: }}
    \begin{equation}
    \label{eqn: dual problem 2}
        \p{\bA^* - \bG_p\bX}\bLambda + \bLambda\p{\bA-\bX\bG_p} = -\bW
    \end{equation}
    \textrm{\textbf{Optimality Condition: }}
    \begin{equation}
    \label{eqn: optimality condition 2}
    p = \frac{1}{\norm{\bX\bLambda\bX}_{\sL(H)}}\br{\bd\bG_p^*\bd\bG_p}^{-1} \bd\bG_p^* \bX\bLambda\bX\bd\bG_p p
    \end{equation}
    \textrm{\textbf{Control Operator Trace Constraint: }}
    \begin{equation}
    \label{eqn: operator trace constraint}
        \trace{\bG_p} = \gamma + \frac{1}{\beta}\norm{\bX\bLambda\bX}_{\sL(H)}
    \end{equation}
\end{subequations}
Because the derivation of \eqref{eqn: primal problem 2} and \eqref{eqn: dual problem 2} is nearly identical to that of \eqref{eqn: primal problem} and \eqref{eqn: dual problem}, we will only discuss the derivation of \eqref{eqn: optimality condition 2} and \eqref{eqn: operator trace constraint} in the following. 

\begin{remark} \label{remark: limit}
An inspection of \eqref{eqn: operator trace constraint} indicates that $\trace{\bG_p} \rightarrow \gamma$ as $\beta \rightarrow +\infty$. Applying \eqref{eqn: operator trace constraint} in \eqref{eqn: cost functional 2} then indicates that $\mathcal J_\beta(\bX;p)$ becomes
$\trace{\bX\bW}$ as $\beta \rightarrow +\infty$. 
\end{remark}

\subsubsection{Derivation of the Optimality Condition}
The chain rule and necessary stationary implies that
$$
\frac{\partial \mathcal L}{\partial p}(q) := \frac{\partial \mathcal L}{\partial \bG_p}\br{\frac{\partial \bG_p}{\partial p}(q)} = 0
$$
for all $q \in \sP$.Because we have assumed that $\bd\bG_p \neq 0$ in \eqref{eqn: nonzero assumption}, we have that 
\begin{equation}
\label{eqn: derivative part 1}
\begin{aligned}
    \frac{\partial \mathcal L}{\partial \bG_p}(\mathbf H) &:= \beta\dual{\mathbf I, \mathbf H} \br{\trace{\bG_p} - \gamma} - \dual{\bLambda , \bX\mathbf H \bX} \\
    &=\beta[\trace{\bG_p}-\gamma]\dual{\mathbf I, \mathbf H} - \dual{\bX\bLambda\bX,\mathbf H} \\
    &= 0
\end{aligned}
\end{equation}
for all $\mathbf H \in \sJ_1(H)$. This then implies that 
$$
    \beta\br{\trace{\bG_p} - \gamma}\mathbf I = \bX\bLambda\bX
$$
and \eqref{eqn: operator trace constraint} follows after norming both sides of the above equation with the $\sL(H)$ norm and performing some algebraic manipulation. 

Taking $\mathbf H = \bd\bG_p(q)$ for all $q \in \sP$ in \eqref{eqn: derivative part 1} results in 
$$
\begin{aligned}
    \frac{\partial \mathcal L}{\partial p}(\bd\bG_p(q))
    &=\beta[\trace{\bG_p}-\gamma]\dual{\mathbf I, \bd\bG_p(q)} - \dual{\bX\bLambda\bX,\bd\bG_p(q)} \\
    &=\beta\br{\trace{\bG_p} - \gamma}
    \p{\bd\bG_p^*\mathbf I, q}_\sP -
    \p{\bd\bG_p^*\bX\bLambda\bX,q}_{\sP} \\
    &= 0.
\end{aligned}
$$
This then implies that
$$
    \bd\bG^*_p\mathbf I = \frac{1}{\beta\br{\trace{\bG_p} - \gamma}} \bd\bG_p^*\bX\bLambda\bX.
$$
We now apply \eqref{eqn: operator trace constraint} to obtain
$$
    \bd\bG^*_p\mathbf I = \frac{1}{\norm{\bX\bLambda\bX}_{\sL\p{H}}} \bd\bG_p^*\bX\bLambda\bX.
$$
Right-multiplying this equation on both sides by $\bd\bG_p p$ then yields
$$
    \bd\bG^*_p\bd\bG_pp = \frac{1}{\norm{\bX\bLambda\bX}_{\sL\p{H}}} \bd\bG_p^*\bX\bLambda\bX\bd\bG_pp
$$
and \eqref{eqn: optimality condition 2} is obtained by left-multiplying both sides of the above by $\br{\bd\bG^*_p\bd\bG_p}^{-1}$ after recalling \eqref{eqn: invertability assumption}. With this, we are ready to analyze the first-order optimality system \eqref{eqn: optimality system 2}.

\subsubsection{Unique Solution to First-Order Optimality System}
We begin our analysis by noticing that there is a continuous mapping $p \mapsto\p{\bX(p), \bLambda(p)}$ where $\bX \in \sJ_1^s(H)$ and $\bLambda\in\sL^s(H)$ correspond to the solutions to \eqref{eqn: primal problem 2} and \eqref{eqn: dual problem 2} respectively. This observation is made using the same argument as in \S\ref{subsubsec: well-posedness 1}. With this, we may focus on determining that \eqref{eqn: optimality condition 2} has a unique fixed point $p \in \sP$ on the manifold induced by \eqref{eqn: operator trace constraint}. 

Let us write $f(\cdot): \sP \rightarrow \sP$ as the nonlinear function defined in \eqref{eqn: optimality condition 2}, i.e.
$$
    f(p) := \frac{1}{\norm{\bX\bLambda\bX}_{\sL(H)}}\br{\bd\bG_p^*\bd\bG_p}^{-1} \bd\bG_p^* \bX(p)\bLambda(p)\bX(p)\bd\bG_p p.
$$
It is our job again to determine the conditions under which there exists a positive constant $k < 1$ that satisfies
$$
    \norm{f(p_1) - f(p_2)}_{\sP} \leq k \norm{p_1 - p_2}_{\sP}
$$
for any two $p_1, p_2 \in \sP$. 

Notice that \eqref{eqn: optimality condition 2} is scale-invariant. Let us take 
\begin{equation} \label{eqn: affine transformation}
p:= s\widetilde p
\end{equation}
where $s\in \mathbb R_+$ is a positive scaling factor and $\widetilde p$ is a reference coordinate space defined such that $\sup \norm{\widetilde p}_\sP = 1$. The scale-invariance arises from the observation that, no matter how one chooses the scaling factor $s$ in the affine transformation in \eqref{eqn: affine transformation}, it never appears in the definition of the fixed-point problem \eqref{eqn: optimality condition 2}. This is because the scaling factors $s$ cancel out in \eqref{eqn: optimality condition 2} when applying this affine coordinate transformation. This then allows us to take $\sup \norm{p}_{\sP} = 1$ in our analysis without loss of generality. 

We now begin our well-posedness analysis by partitioning the fixed point equation into the following
$$
    p = \br{(I)(II)(III)}p,
$$
where we have denoted
$$
    (I) := \frac{1}{\norm{\bX\bLambda\bX}_{\sL(H)}}, \quad
    (II) := \br{\bd\bG^*_p\bd\bG_p}^{-1}, \quad 
    (III) := \bd\bG_p^*\bX\bLambda\bX\bd\bG_p.
$$
These partitioned terms are bounded as follows. 
\begin{equation}
    |(I)| \leq \frac{1}{\mu},
\end{equation}
where $\mu := \inf_{p \in \sP} \norm{\bX(p)\bLambda(p)\bX(p)}_{\sL(H)}$. 

The lower bound of $\mu$ is bounded above by zero since $\bX = \mathbf 0$ if and only if $\bQ = 0$ and $\bLambda \neq \mathbf 0$ as long as $\bW \neq \mathbf \mathbf 0$. To see this, consider the Bochner integral form of the operator-valued Riccati equation \eqref{eqn: Bochner ARE}. First, assume that $\bX = \mathbf 0$, then $\bQ$ must be $\mathbf 0$ since $\int_0^{+\infty}\bS(t)\bQ\bS^*(t)dt=\mathbf 0$ if and only if $\bQ = \mathbf 0$. Now assume $\bQ = \mathbf 0$, then $\int_0^{+\infty}\bS(t)\p{\bX\bG_p\bX}\bS^*(t)dt = \mathbf 0$ implies that $\bX = \mathbf 0$. Next, $\bLambda$ satisfies the Bochner integral form of the dual problem \eqref{eqn: adjoint representation}. The linearity of \eqref{eqn: adjoint representation} and well-posedness of the equation then implies that $\bLambda = 0$ if and only if $\bW = 0$. Hence, we must have that $\inf_{p\in\sP} \norm{\bX(p)\bLambda(p)\bX(p)}_{\sL\p{H}} > 0.$

We then define 
$$
    \norm{(II)}_{\sL\p{\sP}} := K.
$$
And finally, 
$$
    \norm{(III)}_{\sL\p{\sP}} \leq C_{\bd\bG}^2 \frac{M^6}{8\alpha^3} \norm{\bQ}_1^2\norm{\bW}_{\sL\p{H}}.
$$
It is again sufficient to assume that $\sup_{p \in \sP}\norm{p}_{\sP} = 1$ due to the scale invariance of the fixed point equation. With this, we determine that 
$$
\begin{aligned}
    \norm{(I)_1(II)_1(III)_1p_1 - (I)_2(II)_2(III)_2p_2}_{\sP} &\leq
    \left| (I)_1 - (I)_2 \right|\norm{(II)_2}_{\sL\p{\sP}}
    \norm{(III)_2}_{\sL\p{\sP}}\norm{p_2}_{\sP} \\
    &\quad + |(I)_1|\norm{(II)_1 - (II)_2}_{\sL\p{\sP}} \norm{(III)_3}_{\sL\p{\sP}} \norm{p_2}_{\sP}\\
    &\quad + |(I)_1|\norm{(II)_1}_{\sL\p{\sP}} \norm{(III)_1 - (III)_2}_{\sL(\sP)} \norm{p_2}_{\sP}\\
    &\quad + |(I)_1|\norm{(II)_1}_{\sL\p{\sP}}\norm{(III)_1}_{\sL\p{\sP}}\norm{p_1 - p_2}_\sP \\
    &\leq \frac{KC^2_{\bd\bG}M^6}{8\alpha}\norm{\bQ}_1^2\norm{\bW}_{\sL\p{H}}
    |(I)_1-(I)_2|\\
    &\quad + \frac{C^2_{\bd\bG} M^6}{8\alpha^3\mu}\norm{\bQ}_1^2  \norm{\bW}_{\sL(H)} \norm{(II)_1 - (II)_2}_{\sL\p{\sP}}  \\
    &\quad + \frac{K}{\mu} \norm{(III)_1 - (III)_2}_{\sL\p{\sP}} \\
    &\quad + \frac{KC_{\bd\bG}^2 M^6}{8\alpha^3\mu} \norm{p_1 - p_2}_{\sP}.
\end{aligned}    
$$
We now bound every difference term in the above inequality. First, see that
$$
\begin{aligned}
    |(I)_1 - (I)_2| &= \left|\frac{1}{\norm{\bX_1\bLambda_1\bX_1}_{\sL\p{H}}} - \frac{1}{\norm{\bX_2\bLambda_2\bX_2}_{\sL\p{H}}} \right| \\
    &\leq \frac{\norm{\bX_2\bLambda_2\bX_2 -\bX_1\bLambda_1\bX_1}_{\sL\p{H}}}{\norm{\bX_1\bLambda_1\bX_1}_{\sL(H)} \norm{\bX_2\bLambda_2\bX_2}_{\sL\p{H}}}\\
    &\leq \frac{1}{\mu^2} \bigg(\norm{\bX_2-\bX_1}_{\sL(H)}\norm{\bLambda_1}_{\sL\p{H}}\norm{\bX_1}_{\sL\p{H}}\\
    &\qquad + \norm{\bX_2}_{\sL\p{H}} \norm{\bLambda_2-\bLambda_1}_{\sL\p{H}}\norm{\bX_1}_{\sL\p{H}} \\
    &\qquad + \norm{\bX_2}_{\sL\p{H}}\norm{\bLambda_2}_{\sL\p{H}}\norm{\bX_2-\bX_1}_{\sL\p{H}}\bigg) \\
    &\leq \frac{M^4}{2\alpha^2\mu^2} \norm{\bQ}_1\norm{\bW}_{\sL\p{H}} \norm{\bX_1 - \bX_2}_1 +  \frac{M^4}{4\alpha^2\mu^2}\norm{\bQ}_1^2\norm{\bLambda_1-\bLambda_2}_{\sL\p{H}} \\
    &\leq \frac{L_\bG M^{10}}{16\alpha^5\mu}\norm{\bQ}_1^3\norm{\bW}_{\sL\p{H}} \norm{p_1-p_2}_{\sP} + 
    \frac{L_\bG M^4}{4\alpha^2\mu^2}\p{\frac{M^{10}\gamma_\beta}{16\alpha^5}\norm{\bQ}^4_1 + \frac{M^6}{4\alpha^3}\norm{\bQ}^3_1}\norm{p_1-p_2}_\sP\\
    &= k_{(I),\alpha, \bQ} \norm{p_1 - p_2}_\sP,
\end{aligned}
$$
where we have defined $\gamma_\beta:= \gamma + \frac{1}{\beta} \sup_{p\in\sP}\norm{\bX(p)\bLambda(p)\bX(p)}_{\sL\p{H}}$. 
Next, from the Lipshitz continuity of $\bd\bG_{(\cdot)}$ with respect to $p \in \sP$, we have that
$$
    \norm{(II)_1-(II)_2}_{\sL\p{\sP}} \leq 
    L \norm{p_1 - p_2}_{\sP}.
$$
Finally, we derive a Lipschitz bound for the final partitioned difference term. Leveraging the analysis presented in the previous section, we have that
$$
\begin{aligned}
    &\norm{(III)_1 - (III)_2}_{\sL\p{H}}\\ &\quad\leq 
    \norm{\p{\bd\bG_{p_1}^*\bX_1\bLambda_1\bX_1 - \bd\bG_{p_2}^*\bX_2\bLambda_2\bX_2}\bd\bG_{p_2}}_{\sL\p{\sP}} + \norm{\bd\bG_{p_1}^*\bX_1\bLambda_1\bX_1\p{\bd\bG_{p_1}-\bd\bG_{p_2}}}_{\sL\p{H}} \\
    &\quad\leq C_{\bd\bG}\bigg(
         \frac{L_{\bd\bG}M^6}{16\alpha^3}
        \norm{\bQ}_1^2\norm{\bW}_{\sL(H)}
        \norm{p_1 - p_2}_\sP\\
         &\qquad + \frac{L_{\bd\bG}C_{\bd\bG}M^{10}}{16\alpha^5}\norm{\bQ}_1^3 \norm{\bW}_{\sL(H)}\norm{p_1 - p_2}_\sP \\
         &\qquad + 
         \frac{L_{\bG}C_{\bd\bG}M^4}{2\alpha^2}\norm{\bQ}_1^2 
          \p{\frac{M^{10}\gamma_\beta}{16\alpha^5}\norm{\bQ}^2_1 
        + \frac{M^6}{4\alpha^3}\norm{\bQ}_1 }
        \norm{p_1 - p_2}_\sP\bigg) \\
        &\qquad +
        \frac{C_{\bd\bG}L_{\bd\bG}M^6}{8\alpha^3}\norm{\bQ}_1^2 \norm{\bW}_{\sL\p{H}}\norm{p_1 - p_2}_\sP \\
        &= k_{(III), \alpha, bQ}\norm{p_1 - p_2}_{\sP} .
\end{aligned}
$$

Putting this all together, we have determined that there exists a constant $k_{M,\alpha, \beta, \bG,\bQ,\bW} \in \mathbb R_+$ dependent on $M, \alpha, \beta, \norm{\bG}_1, \norm{\bQ}_1, \norm{\bW}_{\sL(H)}$ so that
$$
    \norm{f(p_1) - f(p_2)}_{\sP} \leq k_{M, \alpha, \beta, \bG, \bQ, \bW} \norm{p_1 - p_2}_{\sP}
$$
for any two $p_1, p_2 \in \sP$. Inspection of the definition of the constant $k_{M, \alpha, \beta, \bG, \bQ, \bW}$ indicates that if $\alpha, \beta\in \mathbb R_+$ is sufficiently large, and $\gamma,\norm{\bQ}_1, \norm{\bW} \in \mathbb R_+$ is sufficiently small, then \eqref{eqn: optimality condition 2} has a unique fixed point. 

Evaluating the second variation of $\mathcal L(\cdot,\cdot,\cdot)$ at $(\bX_{opt}, p_{opt}, \bLambda_{opt})$ in the direction vectors contained in the critical cone $\mathcal K(\bX_{opt}, p_{opt})$, as defined in \eqref{eqn: critical cone}, then yields 
$$
\begin{aligned}
    \bd^2\mathcal L_{opt}[(\bPhi, q), (\bPhi, q)] &= 
    \beta\br{\trace{\bG_{p_{opt}}} - \gamma}\dual{\mathbf I, \bd^2\bG_{p_{opt}}(q,q)}\\
    &\qquad+ \beta\dual{\mathbf I, \bd\bG_{p_{opt}}(q)}^2
    - \dual{\bLambda_{opt},\bX_{opt}\bd^2\bG_{p_{opt}}(q,q)\bX_{opt}} \\
    &= \norm{\bX_{opt}\bG_{p_{opt}} \bX_{opt}}_{\sL\p{H}}\dual{\mathbf I, \bd^2\bG_{p_{opt}}(q,q)}\\
    &\qquad+ \beta\dual{\mathbf I, \bd\bG_{p_{opt}}(q)}^2
    - \dual{\bLambda_{opt},\bX_{opt}\bd^2\bG_{p_{opt}}(q,q)\bX_{opt}} \\
\end{aligned}
$$
for all $\p{\bPhi, q} \in \mathcal K\p{\bX_{opt}, p_{opt}}.$ Inspecting the expression above indicates that $\bd^2\mathcal L_{opt}[(\bPhi, q), (\bPhi, q)]$ is positive definite if $\beta \in \mathbb R_+$ is sufficiently large after recalling \eqref{eqn: operator trace constraint}. 

\begin{theorem} \label{theorem: well-posedness optimization 2}
Assume that $\bA: \mathcal D(\bA) \rightarrow H$ is the generator of a $C_0$-semigroup, $\bQ \in \sJ_1^s(H)$, and $\bW \in \sL^s(H)$. Then, if conditions \eqref{eqn: G bound}, \eqref{eqn: G lipschitz bound}, \eqref{eqn: dG lipshitz bound}, and \eqref{eqn: invertability assumption} are satisfied for the mapping $\bG_{(\cdot)}: \sP \rightarrow \sJ_1^s(H)$, then there exists a unique solution $\p{\bX_{opt}, \bLambda_{opt}, p} \in \sJ_1^s(H) \times \sL^s\p{H} \times \sP$ that satisfies the first-order optimality system \eqref{eqn: optimality system 2} associated with the constrained weighted trace minimization problem, provided that the penalty parameter $\beta\in\mathbb R_+$ is chosen sufficiently large enough and that $\alpha \in \mathbb R_+$ is sufficiently large to reduce the contributions of $\norm{\bQ}_1, \norm{\bW}_{\sL(H)},$ and $M$ in the definition of the constant $k_{M, \alpha, \beta, \bG, \bQ, \bW}$.

Further assume that \eqref{eqn: bounded second variation} is satisfied for the mapping $\bG_{(\cdot)}: \sP \rightarrow \sJ_1^s(H)$. Then $\p{\bX_{opt}, p_{opt}} \in \sJ_1^s(H)\times \sP$ is the unique constrained minimizer for the cost functional \eqref{eqn: cost functional 2} provided that the conditions described above for uniqueness are satisfied. 
\end{theorem}

Theorem \ref{theorem: well-posedness optimization 2} provides the necessary conditions needed for $\mathcal J_\beta(\cdot)$ to have a constrained minimizer. In general, the constrained minimizer is not guaranteed to be unique unless $\alpha$ is sufficiently large enough to reduce the effects of $\norm{\bQ}_1, \norm{\bW}_{\sL\p{H}}$, and $M$ in the definition of the contraction constant $k_{M, \alpha, \beta, \bG, \bQ, \bW}$. This is in contrast to the constrained optimization problem presented in \S\ref{sec: problem 1} where $\beta$ can be chosen large enough to guarantee the uniqueness. 

Assume for now that the conditions in Theorem \ref{theorem: well-posedness optimization 2} are satisfied and $\alpha$ is sufficiently large enough to guarantee the uniqueness of the fixed point for $\beta$ chosen large enough. Let us now define $(\bX^\beta_{opt}, p^\beta_{opt})\in \sJ_1^s(H) \times \sP$ to be the constrained minimizer associated with the choice of $\beta$ in the definition of \eqref{eqn: cost functional 2}. Theorem \ref{theorem: well-posedness optimization 2} suggests that there is a unique constrained minimizer to $\mathcal J_\beta(\cdot)$ for every $\beta > \beta_{min}$ for some threshold value $\beta_{min}$. Taking $\beta \rightarrow +\infty$ preserves the contractive property of the fixed point problem \eqref{eqn: optimality condition 2} while simultaneously enforcing the constraint $\trace{\bG_p} = \gamma$. This then implies that the limit $\lim_{\beta\rightarrow \infty} \p{\bX^\beta_{opt}, p^\beta_{opt}} = \p{\bX^\infty_{opt}, p^\infty_{opt}}$ is well-defined. We also observe that the Hessian operator, while remaining positive, becomes unbounded in this limit. This suggests that the stationary point associated with the constrained minimizer $\p{\bX^\infty_{opt}, p^\infty_{opt}}$ is not twice differentiable. This discussion along with Remark \ref{remark: limit} then implies the following.

\begin{corollary}
Assume that the conditions in Theorem \ref{theorem: well-posedness optimization 2} are satisfied. Then there exists a unique constrained minimizer $\p{\bX_{opt},p_{opt}} \in \sJ_1^s(H) \times \sP$ that satisfies the following constrained minimization problem. 
\begin{equation*}
    \min_{p \in \sP} \trace{\bX(p)\bW}
\end{equation*}
subject to
\begin{equation*}
\left\{
\begin{aligned}
    \bA\bX + \bX\bA^* - \bX\bG_p\bX + \bQ &= \mathbf 0 \\
    \trace{\bG_p} &= \gamma.
\end{aligned}
\right.
\end{equation*}
\end{corollary}

\section{Discussion} \label{sec: discussion}


In this work, we have determined the well-posedness of the strong form of the operator-valued Riccati equation (the primal problem). Upon this foundation, we are able to determine the well-posedness of the dual problem and prove that the solutions to both the primal and dual problems are Lipschitz continuous with respect to the parameter variable that describes the operator associated with the control device. From there, we are able to determine that both constrained weighted trace minimization problems presented in this work have unique constrained minimizers under certain necessary conditions through a fixed-point argument. In the penalization parameter can be chosen sufficiently in the penalized parameter optimization problem to force uniqueness, whereas this property does not extend to the problem of penalization for approximate constraint enforcement. Under the conditions in which the second optimization problem does have a unique minimizer, we determine that in the infinite limit of the penalization parameter that the uniqueness of the minimizer is preserved. This then implies that exact constraint enforcement also leads to a unique minimizer provided that the necessary conditions prescribed in Theorem \ref{theorem: well-posedness optimization 2} are satisfied. 

The results of this work are mostly theoretical. However, it can be utilized to design efficient sensor and actuator design and placement algorithms. The uniqueness of the minimizer enables the use of gradient-based optimization algorithms, as opposed to more expensive global optimization algorithms, since uniqueness is known a priori, The results regarding the strong form of the operator-valued Riccati equation indicates that its solution is compact in the sense that it maps $\mathcal D(\bA^*)'$ into $\mathcal D(\bA^*)$. We plan on utilizing this observation in an attempt to derive optimal convergence rates for the numerical approximation for operator-valued Riccati equations that arise from unbounded sensing and actuation problems in a future work. 

\bibliographystyle{plain}
\bibliography{bibliography}

\begin{thebibliography}{10}

\bibitem{axler2024linear}
Sheldon Axler.
\newblock {\em Linear algebra done right}.
\newblock Springer Nature, 2024.

\bibitem{bensoussan2006optimization}
Alain Bensoussan.
\newblock Optimization of sensors' location in a distributed filtering problem.
\newblock In {\em Stability of Stochastic Dynamical Systems: Proceedings of the International Symposium Organized by “The Control Theory Centre”, University of Warwick, July 10--14, 1972 Sponsored by the “International Union of Theoretical and Applied Mechanics”}, pages 62--84. Springer, 2006.

\bibitem{bensoussan2007representation}
Alain Bensoussan, Giuseppe Da~Prato, Michel~C Delfour, and Sanjoy~K Mitter.
\newblock {\em Representation and control of infinite dimensional systems}.
\newblock Springer, 2007.

\bibitem{burns2015infinite}
John~A Burns and Carlos~N Rautenberg.
\newblock The infinite-dimensional optimal filtering problem with mobile and stationary sensor networks.
\newblock {\em Numerical Functional Analysis and Optimization}, 36(2):181--224, 2015.

\bibitem{burns2015solutions}
John~A Burns and Carlos~N Rautenberg.
\newblock Solutions and approximations to the riccati integral equation with values in a space of compact operators.
\newblock {\em SIAM Journal on Control and Optimization}, 53(5):2846--2877, 2015.

\bibitem{cheung2025approximation}
James Cheung.
\newblock On the approximation of operator-valued riccati equations in hilbert spaces.
\newblock {\em Journal of Mathematical Analysis and Applications}, 547(1):129250, 2025.

\bibitem{ciarlet2025linear}
Philippe~G Ciarlet.
\newblock {\em Linear and nonlinear functional analysis with applications}.
\newblock SIAM, 2025.

\bibitem{conway2025course}
John~B Conway.
\newblock {\em A course in operator theory}, volume~21.
\newblock American Mathematical Society, 2025.

\bibitem{curtain2012introduction}
Ruth~F Curtain and Hans Zwart.
\newblock {\em An introduction to infinite-dimensional linear systems theory}, volume~21.
\newblock Springer Science \& Business Media, 2012.

\bibitem{diestel1974vector}
J.~Diestel and J.~J. Uhl.
\newblock {\em Vector Measures}.
\newblock American Mathematical Society, 1977.

\bibitem{edalatzadeh2019optimal}
M~Sajjad Edalatzadeh, Dante Kalise, Kirsten~A Morris, and Kevin Sturm.
\newblock Optimal actuator design for vibration control based on lqr performance and shape calculus.
\newblock {\em arXiv preprint arXiv:1903.07572}, 2019.

\bibitem{goldstein2017semigroups}
Jerome~A Goldstein.
\newblock {\em Semigroups of linear operators and applications}.
\newblock Courier Dover Publications, 2017.

\bibitem{hintermuller2017optimal}
Michael Hinterm{\"u}ller, Carlos~N Rautenberg, Masoumeh Mohammadi, Martin Kanitsar, et~al.
\newblock Optimal sensor placement: A robust approach.
\newblock {\em SIAM J. Control. Optim.}, 55(6):3609--3639, 2017.

\bibitem{hu2016sensor}
Weiwei Hu, Kirsten Morris, and Yangwen Zhang.
\newblock Sensor location in a controlled thermal fluid.
\newblock In {\em 2016 IEEE 55th Conference on Decision and Control (CDC)}, pages 2259--2264. IEEE, 2016.

\bibitem{kirsch2011introduction}
Andreas Kirsch et~al.
\newblock {\em An introduction to the mathematical theory of inverse problems}, volume 120.
\newblock Springer, 2011.

\bibitem{5482053}
Kirsten Morris.
\newblock Linear-quadratic optimal actuator location.
\newblock {\em IEEE Transactions on Automatic Control}, 56(1):113--124, 2011.

\bibitem{morris2015comparison}
Kirsten Morris and Steven Yang.
\newblock Comparison of actuator placement criteria for control of structures.
\newblock {\em Journal of Sound and Vibration}, 353:1--18, 2015.

\bibitem{sharrock2022joint}
Louis Sharrock and Nikolas Kantas.
\newblock Joint online parameter estimation and optimal sensor placement for the partially observed stochastic advection-diffusion equation.
\newblock {\em SIAM/ASA Journal on Uncertainty Quantification}, 10(1):55--95, 2022.

\bibitem{troltzsch2010optimal}
Fredi Tr{\"o}ltzsch.
\newblock {\em Optimal control of partial differential equations: theory, methods, and applications}, volume 112.
\newblock American Mathematical Soc., 2010.

\end{thebibliography}
\end{document}